\documentclass[english,11pt,leqno,twoside]{amsart}
\usepackage[T1]{fontenc}
\usepackage[latin9]{inputenc}
\usepackage{amsmath}
\usepackage{amsthm}
\usepackage{amssymb}

\usepackage{soul,color}
\usepackage{t1enc}
\usepackage{a4,indentfirst,latexsym,graphics,graphicx}
\usepackage{graphics}
\usepackage{mathrsfs}
\usepackage{cite,enumitem,graphicx}
\usepackage[colorlinks=true,urlcolor=blue,
citecolor=red,linkcolor=blue,linktocpage,pdfpagelabels,
bookmarksnumbered,bookmarksopen]{hyperref}
\usepackage[english]{babel}
\usepackage[left=2.1cm,right=2.1cm,top=2.72cm,bottom=2.72cm]{geometry}
\usepackage[metapost]{mfpic}
\usepackage[hyperpageref]{backref}
\usepackage[colorinlistoftodos]{todonotes}
\usepackage[normalem]{ulem}

\makeatletter
\providecommand\@dotsep{5}
\def\listtodoname{List of Todos}
\def\listoftodos{\@starttoc{tdo}\listtodoname}
\makeatother

\numberwithin{equation}{section}
\newtheorem{theorem}{Theorem}[section]
\newtheorem{lemma}[theorem]{Lemma}

\newtheorem{definition}[theorem]{Definition}
\newtheorem{proposition}[theorem]{Proposition}
\newtheorem{remark}[theorem]{Remark}

\title[Nonlinear Schr\"odinger-Bopp-Podolsky-Proca system]{Multiple solutions and profile description for a nonlinear Schr\"odinger-Bopp-Podolsky-Proca system on a manifold}

\author[P. d'Avenia]{Pietro d'Avenia}

\address[P. d'Avenia]{\newline\indent
	Dipartimento di Meccanica, Matematica e Management 
	\newline\indent
	Politecnico di Bari
	\newline\indent
	Via E. Orabona 4
	\newline\indent
	70125 Bari, Italy.}
\email{\href{mailto:pietro.davenia@poliba.it}{pietro.davenia@poliba.it}}

\author[M.G. Ghimenti]{Marco G. Ghimenti}

\address[M.G. Ghimenti]{\newline\indent
	Dipartimento di Matematica 
	\newline\indent
	Universit\`a di Pisa
	\newline\indent
	Largo B. Pontecorvo, 5
	\newline\indent
	56126 Pisa, Italy.}
\email{\href{mailto:marco.ghimenti@unipi.it}{marco.ghimenti@unipi.it}}

\thanks{The authors are members of GNAMPA (INdAM).\\
	Pietro d'Avenia is partially supported by PRIN 2017JPCAPN {\em Qualitative and quantitative aspects of nonlinear PDEs} and by GNAMPA project {\em Modelli EDP nello studio problemi della fisica moderna}.
	Marco G. Ghimenti is partially supported by GNAMPA project {\em Modelli matematici con singolarit\`a per fenomeni di interazione}.}
\date{\today}

\begin{document}

\begin{abstract}
We prove a multiplicity result for
\begin{equation*}
	\begin{cases}
		-\varepsilon^{2}\Delta_g u+\omega u+q^{2}\phi u=|u|^{p-2}u\\
		-\Delta_g \phi +a^{2}\Delta_g^{2} \phi + m^2 \phi =4\pi u^{2}
	\end{cases}
	\text{ in }M,
\end{equation*}
where $(M,g)$ is a smooth and compact $3$-dimensional Riemannian manifold without boundary, $p\in(4,6)$, $a,m,q\neq 0$, $\varepsilon>0$ small enough. The proof of this result relies on Lusternik-Schnirellman category. We also provide a profile description for low energy solutions.
\end{abstract}

\maketitle


\section{Introduction}

In this paper we study the system
\begin{equation}\label{eq:BPP}
	\begin{cases}
	-\varepsilon^{2}\Delta_g u+\omega u+q^{2}\phi u=|u|^{p-2}u\\
	-\Delta_g \phi +a^{2}\Delta_g^{2} \phi + m^2 \phi =4\pi u^{2}
	\end{cases}
\text{ in }M,
\end{equation}
where $(M,g)$ is a smooth and compact $3$-dimensional Riemannian manifold without boundary, $\Delta_g$ is the Laplace-Beltrami operator, $u,\phi:M\to\mathbb{R}$, $p\in(4,6)$, $a,m,q\neq 0$, $\varepsilon>0$.\\
System \eqref{eq:BPP} can be obtained starting from the classical nonlinear Schr\"odinger Lagrangian density
\begin{equation}\label{Schlagrdens}
\mathcal{L}_{\rm S}(\psi)
:=i\hbar \partial_t \overline{\psi}
-\frac{\hbar^2}{2m_0^2} |\nabla_g \psi|^2
+\frac{2}{p}|\psi|^p,
\end{equation}
with $\psi:\mathbb{R}\times M \to \mathbb{C}$, $\hbar>0$ (the Plank constant), and $m_0\neq 0$.\\
Such a Lagrangian density describes a charged particle $\psi$ and, to study it in the electromagnetic field $({\bf E},{\bf B})$ generated by itself motion, it is usual to apply the {\em minimal coupling rule}. This consists into replacing in \eqref{Schlagrdens} the usual temporal and spatial derivatives $\partial_t, \nabla_g$, with the gauge covariant ones
$$\partial_t + i \frac{q}{\hbar}\phi,
\quad
\nabla_g - i \frac{q}{\hbar c}{\bf A},$$
where $(\phi,{\bf A})$ is the gauge potential related to $({\bf E},{\bf B})$, $q$ is a coupling constant, and $c$ is the speed of light, obtaining
\[
\mathcal{L}_{\rm Coupl} (\psi,\phi,{\bf A})
:=
i\hbar \partial_t \overline{\psi} - q \phi |\psi|^2
-\frac{\hbar^2}{2m_0^2} |\nabla_g \psi-i\frac{q}{\hbar c}{\bf A} \psi |^2
+\frac{2}{p}|\psi|^p.
\]
To this Lagrangian density we have to add the electromagnetic field one.
Of course this step implies the choice of an electromagnetic theory.
In our case we consider the Bopp-Podolsky one in the Proca setting that is
\begin{align*}
\mathcal{L}_{\rm BPP}(\phi,{\bf A})
&:=
\frac{1}{8\pi}
\left\{
|\nabla_g \phi + \frac{1}{c} \partial_t {\bf A}|^2
- |\nabla_g \times {\bf A}|^2
+ m^2 (|\phi|^2-|{\bf A}|^2)
\right.\\
&\qquad\qquad
\left.
+ a^2
\Big[
(
\Delta_g \phi + \frac{1}{c} \nabla_g \cdot \partial_t {\bf A}
)^2
-
|
\nabla_{g}\times\nabla_{g}\times {\bf A}
+\frac{1}{c} \partial_t (\nabla_g \phi
+ \frac{1}{c} \partial_t{\bf A}
)
|^2
\Big]
\right\}.
\end{align*}
It was introduced, without the Proca term, namely with $m=0$, independently in \cite{B} and \cite{P} as an higher-order perturbation of the classical Maxwell theory, to solve its {\em infinity problem}: the energy of the electromagnetic field generated by a pointwise charge, in the electrostatic case, is not finite.\\
Hence, the Euler-Lagrange equations for the total action
\[
\mathcal{S}_{\rm tot} (\psi,\phi,{\bf A}):=
\iint [\mathcal{L}_{\rm Coupl} (\psi,\phi,{\bf A})
+\mathcal{L}_{\rm BPP}(\phi,{\bf A})] d\mu_g dt
\]
in the purely electrostatic case, for standing waves $\psi(t,x)=e^{i\omega t} u(x)$, normalizing some parameter, and applying some rescaling, give \eqref{eq:BPP}. For more details we refer the reader to \cite{DS,H20}.

In the last few years a wide literature on this topic is developing both in $\mathbb{R}^3$ (see \cite{CLRT,CT,FSZamp,FS,LPT,MS,PJ,KJDE,Z} and references therein) and on manifolds, due to Hebey (see \cite{HebDCDS,H20,HebCCM,HebADE}).

In this paper we want to prove the  following multiplicity result for \eqref{eq:BPP} in the spirit of the semiclassical limit.
\begin{theorem}\label{th:main}
	Let $4 < p < 6$ and $2am<1$.
	For $\varepsilon$ small enough we have at least $\operatorname{cat}(M)$ nonconstant positive solutions $(u_\varepsilon,\phi_\varepsilon )$ of (\ref{eq:BPP}) with {\em low energy}. 
	 The functions $u_\varepsilon$ have a unique maximum point $P_\varepsilon$ and 
	$u_\varepsilon = W_{\varepsilon,P_\varepsilon} +R_\varepsilon$ where $W_{\varepsilon,P_\varepsilon}$ is defined in \eqref{Wxieps} and $|R_\varepsilon|_{\infty}\rightarrow 0$ as $\varepsilon\rightarrow 0$. 
	In addition, $\|\phi_\varepsilon\|_{C^2(M)}\rightarrow 0$ as $\varepsilon\rightarrow 0$. 
	Finally, there exists at least a further nonconstant positive solution of  (\ref{eq:BPP}) with {\em higher energy}.
\end{theorem}

Here $\operatorname{cat}(M)$ is the Lusternik-Schnirellman category which is hereafter defined.

\begin{definition}
	\label{def:cat}Let $X$ a topological space and consider a closed
	subset $A\subset X$. We say that $A$ has category $k$ relative
	to $X$ ($\operatorname{cat}_{X}A=k$) if $A$ is covered by $k$ closed sets $A_{j}$,
	$j=1,\dots,k$, which are contractible in $X$, and $k$ is the minimum
	integer with this property. We simply denote $\operatorname{cat} X=\operatorname{cat}_{X}X$.\end{definition}

The proof relies on a topological method, sometimes called {\em photography method}, which was firstly introduced for 
an elliptic nonlinear critical equation in a bounded domain by Bahri and Coron  \cite{BaC}, then readapted by Benci and Cerami \cite{BC} and by Benci, Cerami, and Passaseo \cite{BCP}, and thereafter used in a wide class of elliptic problems, both on domains and on manifolds. 

The main idea of this method is to establish a connection between the functions which have low energy and  the domain of the equation. Roughly speaking, on the one hand it is possible to construct, for any point of the domain, a function which is peaked in this point, and, on the other hand, it can be showed that any function, for which the energy functional is sufficiently small, is concentrated around a point of the domain. Once this link is set, one can prove that the topology of the low energy functions is at least rich as the topology of the domain, and use classical results to link this topology to the number of solutions. This methods are particularly useful on manifolds, since, for example, for any compact smooth manifold without boundary it holds $\operatorname{cat} M \ge 2$.

To prove that a low energy function is concentrated around a point of the manifold, we have to consider manifolds smoothly embedded in some euclidean space $\mathbb{R}^N$, for a suitable $N$. This does not imply a loss of generality: such an embedding exists for any smooth compact Riemannian manifold. 

In the context of Riemannian manifolds, this method has been used for the nonlinear Klein-Gordon-Maxwell-Proca system in \cite{GM} (see  references therein for different situations).

A similar topological method has also been used by Siciliano, Figueiredo, and Mascaro in \cite{FS,MS} to obtain multiplicity result of the doubly perturbed nonlinear Schr\"odinger-Bopp-Podolsky system
\begin{equation*}
	\begin{cases}
	-\varepsilon^{2}\Delta u+V(x) u+\lambda\phi u=f(u)\\
	-\varepsilon^{2}\Delta \phi +\varepsilon^{4}\Delta^{2} \phi = u^{2}
	\end{cases}
\text{ in }\mathbb{R}^3.
\end{equation*}
The number of solutions is related to the Lusternik-Schnirelmann category of the set $\{x\in \mathbb{R}^3\ :\ V(x)=\min V\}$.


We would like to make a final comment on the exponent $p\in(4,6)$. In our setting, this condition is necessary since one of our main tools is the Nehari manifold: it is a natural constraint and, for $p>4$, it is smooth and the energy functional on such a constraint is bounded from below.
Dealing with systems in $\mathbb{R}^3$, the condition on $p$ tipically could be loosened using,  for example, other additional natural constraints, as the Pohozaev constraint or a {\em linear combination} of Nehari and Pohozaev identities (see \cite{DS}). In the Riemannian manifold setting, Pohozaev type identities carry on extra terms and they can be hard to handle.

Our paper is organized as follows. Section \ref{sec:prelim} contains a list of useful results and all the preliminary definitions. In Section  \ref{sec:funct} the variational framework of the problem is presented, while the rest of the paper is devoted to the proof of  Theorem \ref{th:main}. In particular, the connection between low energy function and the manifold $M$ is proved in Section \ref{sec:low}, as well as the main topological result to obtain $\operatorname{cat}(M)$ low energy solutions. These solutions turn out to be nontrivial and a description of their profile is given in Section \ref{sec:prof}. Finally we prove that there exists an additional nontrivial solution in Section  \ref{sec:further}, concluding the proof of Theorem \ref{th:main}.\\
For the sake of readability, we have collected two technical and somewhat classical proof in the Appendix.



\section{Preliminaries}\label{sec:prelim}
From now on, for the sake of simplicity, we will take $\omega=q=m=1$ in \eqref{eq:BPP}.

In the following we will use the notation
$$\|v\|_{H^2}^{2}:=\int_{M}(a^{2}|\Delta_g v|^{2}+|\nabla_g v|^{2}+v^{2})d\mu_{g},
\quad
\|v\|_{H^{1}}^{2}:=\int_{M}(|\nabla_g v|^{2}+ v^{2})d\mu_{g},
\quad
|v|_{p}^{p}  :=\int_{M}|v|^{p}d\mu_{g}
$$
for the norms on $H^{2}(M)$, $H^1(M)$, and $L^p(M)$, and, with abuse of notation, also for the respective norms in $\mathbb{R}^3$. Moreover, for every fixed $\varepsilon>0$, we will use
\[
\|v\|_{\varepsilon}^{2}  :=\frac{1}{\varepsilon}\int_{M}|\nabla_g v|^{2}d\mu_{g}+\frac{1}{\varepsilon^{3}}\int_Mv^{2}d\mu_{g},
\quad
|v|_{p,\varepsilon}^{p}  :=\frac{1}{\varepsilon^{3}}\int_{M}|v|^{p}d\mu_{g}.
\]
We recall that there exists $C>0$, indipendent of $\varepsilon$, such that, for every $p\in [1,6]$,
\begin{equation}\label{imb}
	|v|_{p,\varepsilon}\le C\|v\|_{\varepsilon}.
\end{equation}

Now let us recall some known properties about the second equation in \eqref{eq:BPP}, whose proof can be found in \cite[Lemma 3.1 and Lemma 4.1]{HebDCDS}.\footnote{Without the normalization of the constants, our results hold for every $\omega,q>0$ and $2am<1$.}

\begin{lemma}
\label{lem:BP}For every $u\in H^{1}(M)$ there exists a unique $\phi_{u}\in H^{4}(M)\cap C^2(M)$ solution of 
\begin{equation}
-\Delta_g v+a^{2}\Delta_g^{2}v+v=4\pi u^{2}\label{eq:BP2}
\quad\text{ in } M
\end{equation}
and
\begin{enumerate}[label=(\alph{*}), ref=\alph{*}]
	\item \label{aLem:BP} there exists $C>0$
	such that, for every $u\in H^{1}(M)$, $\|\phi_{u}\|_{H^{2}}\le C |u|_{2}^{2}$ and $\|\phi_{u}\|_{H^{4}}\le C |u|_{4}^{2}$;
	\item \label{bLem:BP} if $a<1/2$, then, for every $u\in H^{1}(M)$, $\phi_u\ge0$.
\end{enumerate}
\end{lemma}
In view of the previous Lemma we write \eqref{eq:BPP} as
\begin{equation}\label{eq:BPPnl}
-\varepsilon^{2}\Delta_g u+ u+\phi_u u=|u|^{p-2}u
\qquad
\text{ in } M.
\end{equation}

Moreover we will use the further results that involve $\phi_u$.
\begin{lemma}\label{lem:BP2}
The map $\Phi:=u\in H^{1}(M) \mapsto \phi_{u}\in H^{2}$ is $C^{2}$ and, for every $u\in H^1(M)$ and $h,k\in H^1(M)$, $\Phi'(u)[h]$ and $\Phi''(u)[h,k]$ are the unique solutions of
	\begin{equation}
		-\Delta_g v+a^{2}\Delta_g^{2}v+v=8\pi uh \quad\text{ in } M \label{eq:der1}
	\end{equation}
and
\begin{equation}
	-\Delta_g v+a^{2}\Delta_g^{2}v+v=8\pi hk \quad\text{ in } M,\label{eq:der2}
\end{equation}
respectively.\\
Moreover, for every $t\in\mathbb{R}$ and $u\in H^{1}(M)$, $\Phi(tu)=t^2\Phi(u)$ and if $\{u_n\}\subset H^1(M)$ converges weakly to $\bar{u}$ in $H^1(M)$, then, up to a subsequence, $\Phi(u_n)\to\Phi(\bar{u})$ in $H^2(M)$.
\end{lemma}

\begin{proof}
Let $\tilde{\phi}$ be the unique solution of \eqref{eq:der1} in $H^2(M)$. Then
\[
\Phi(u+h)-\Phi(u) - \tilde{\phi}=\Phi(h)
\]
since
\[
-\Delta_{g} [\Phi(u+h)-\Phi(u)-\tilde{\phi}]
+a^2 \Delta_{g}^2 [\Phi(u+h)-\Phi(u)-\tilde{\phi}]
+ [\Phi(u+h)-\Phi(u)-\tilde{\phi}]
=4\pi h^2
\]
and, since, by Lemma \ref{lem:BP}, as $h\to 0$ in $H^1(M)$, $\| \Phi(h)\|_{H^2}/\|h\|_{H^1}\leq C \|h\|_{H^1} \to 0$, we get that $\Phi'(u)[h]=\tilde{\phi}$.\\
Analogously, if $\tilde{\psi}$ is the unique solution of \eqref{eq:der2} in $H^2(M)$, then, using that for every $h\in H^1(M)$, $\Phi'(u)[h]$ is the unique solution of \eqref{eq:der1}, we have
\[
-\Delta_{g} [\Phi'(u+k)[h]-\Phi'(u)[h]-\tilde{\psi}]
+a^2 \Delta_{g}^2 [\Phi'(u+k)[h]-\Phi'(u)[h]-\tilde{\psi}]
+ [\Phi'(u+k)[h]-\Phi'(u)[h]-\tilde{\psi}]
=0
\]
and so
\[
\Phi'(u+k)[h]-\Phi'(u)[h]-\tilde{\psi}=0.
\]
This easily implies that $\Phi''(u)[h,k]=\tilde{\psi}$.\\
To show that the map $u\in H^1(M)\mapsto\Phi'(u)$ is continuous we observe that, if $u_n\to u$ in $H^1(M)$, then $\Phi'(u_n)[h]-\Phi'(u)[h]$ is the unique solution of
\[
-\Delta_{g} v +a^2 \Delta_{g}^2 v + v
=
8\pi (u_n - u) h.
\]
Thus
\begin{align*}
\|\Phi'(u_n)[h]-\Phi'(u)[h]\|_{H^2}^2
&=
8\pi \int_M [\Phi'(u_n)[h]-\Phi'(u)[h]](u_n - u) h d\mu_{g}\\
&\leq
C
\|\Phi'(u_n)[h]-\Phi'(u)[h]\|_{H^2}
\|u_n - u\|_{H^1} \|h\|_{H^1}
\end{align*}
that allows us to conclude easily.\\
Analogously we can prove that the map $u\in H^1(M)\mapsto\Phi''(u)$ is continuous using that $\Phi''(u_n)[h]-\Phi'(u)[h]$ is the unique solution of
\[
-\Delta_{g} v +a^2 \Delta_{g}^2 v + v
=
0.
\]
Finally, the last part of the statement follows from
\begin{align*}
-\Delta_{g} \Phi(tu)
+a^2 \Delta_{g}^2 \Phi(tu)
+ \Phi(tu)
&=4\pi t^2 u^2
= t^2
[
-\Delta_{g} \Phi(u)
+a^2 \Delta_{g}^2 \Phi(u)
+ \Phi(u)
]\\
&=
-\Delta_{g} [t^2\Phi(u)]
+a^2 \Delta_{g}^2 [t^2\Phi(u)]
+[t^2 \Phi(u)]
\end{align*}
and observing that, since up to a subsequence, $u_{n}\to\bar{u}$ in $L^{\tau}(M)$ for $1\le \tau<6$, then, for any $\varphi \in H^2(M)$,
\begin{equation}\label{wc}
\langle\Phi(u_n),\varphi\rangle_{H^2}
= 4\pi \int_{M} u_n^2\varphi d\mu_g
\to 4\pi \int_{M} \bar{u}^2\varphi d\mu_g
= \langle\Phi(\bar{u}),\varphi\rangle_{H^2}
\end{equation}
and so $ \{\Phi(u_n)\}$ converges to $\Phi(\bar{u})$ weakly in $H^2(M)$ and, up to a subsequence, strongly in $L^{\tau}(M)$ for $\tau\geq 1$.\\
Moreover, by (\ref{aLem:BP}) in Lemma \ref{lem:BP} and \eqref{wc},
\begin{align*}
| \| \Phi(u_n)\|_{H^2}^2 - \| \Phi(\bar{u})\|_{H^2}^2 |
&=
4\pi 
\left| \int_M u_n^2 \Phi(u_n) d\mu_g - \int_M \bar{u}^2 \Phi(\bar{u}) d\mu_g  \right|
\\
&\leq
4\pi \left[
\int_M  |\Phi(u_n) - \Phi(\bar{u})| u_n^2 d\mu_g
+\left| \int_M \Phi(\bar{u}) [u_n^2 - \bar{u}^2]   d\mu_g  \right|\right]\\
&\leq
4\pi \left[
|\Phi(u_n) - \Phi(\bar{u})|_2 |u_n|_4^2 
+\left| \int_M \Phi(\bar{u}) [u_n^2 - \bar{u}^2]   d\mu_g  \right|\right]\to 0
\end{align*}
and we conclude.
\end{proof}
Now let us consider the functional
$$
G:=u\in H^{1}(M) \mapsto \int_{M}u^{2}\phi_{u}d\mu_{g}
$$
which is well defined and, for every $u\in H^1(M)$, since $\phi_u$ is the unique solution of \eqref{eq:BP2}, 
\begin{equation}
	G(u)=\frac{1}{4\pi}\|\phi_{u}\|_{H^{2}}^{2}\geq 0\label{eq:punto 3}
\end{equation}
and
\[
G(u)=0
\iff
\phi_u=0
\iff
u=0.
\]
It satisfies the following properties.
\begin{lemma}\label{lem:BP3}
For every $u\in H^1(M)$, $|G(u)| \leq C\|\phi_{u}\|_{H^2} \|u\|_{H^1}^2$. Moreover the functional $G$ is $C^{1}$, for every $u,h\in H^1(M)$,
\[
G'(u)[h]=4\int_{M}\phi_{u} u hd\mu_{g},
\]
and, if $\{u_n\}\subset H^1(M)$ converges weakly to $\bar{u}$ in $H^1(M)$, then, up to a subsequence, $G(u_n)\to G(\bar{u})$.
\end{lemma}

\begin{proof}
The first property and the continuity of $G$ are immediate consequences of H\"older inquality and of the previous Lemma.\\
By \eqref{eq:punto 3} and since $\Phi'(u)[h]$ is the unique solution of \eqref{eq:der1}, we have
\[
G'(u)[h]
=
\frac{1}{2\pi}\langle\phi_u,\Phi'(u)[h]\rangle_{H^2}
=
\frac{1}{2\pi}\int_{M} \phi_{u}[-\Delta_g \Phi'(u)[h]+a^{2}\Delta_g^{2}\Phi'(u)[h]+\Phi'(u)[h]]d\mu_{g}
=
4 \int_{M}\phi_{u}uhd\mu_{g}.
\]
The continuity of $G'$ follows from
\[
|G'(u_n)[h]-G'(u)[h]|
\leq
C[\|\phi_{u_n}-\phi_u\|_{H^2}\|u_n\|_{H^1}+\|\phi_{u_n}\|_{H^2}\|u_n-u\|_{H^1}]\|h\|_{H^1}.
\]
Finally, the last part of the statement is an easy consequence of \eqref{eq:punto 3} and of Lemma \ref{lem:BP2}.
\end{proof}

Let us conclude this section with some recall concerning the manifold $M$.
\\
Let us consider the $C^\infty$ exponential map $\exp: TM\to M$. 
Since $M$ is compact, there exists $r>0$, called {\em injectivity radius}, such that $\exp_\xi|_{B(0,r)}:B(0,r)\to B_g(\xi,r)$ is a diffeomorphism for any $\xi\in M$. Fixed $\xi\in M$, for every $y\in B(0,r)=\exp_\xi^{-1}(B_g(\xi,r))$, we have
\begin{equation}\label{asymptg}
	(g_\xi)_{ij}(y)=\delta_{ij} + \frac{1}{3} R_{ihlj} y^h y^l + O(|y|^3),
	\quad
	|g_\xi(y)|:=\operatorname{det}((g_\xi)_{ij})(y)=1-\frac{1}{3} R_{ij} y^i y^j + O(|y|^3)
\end{equation}
where $R_{ij}$ and $R_{ihlj}$ are the components of the Ricci curvature tensor and of the Riemann curvature tensor, respectively (see e.g. \cite{SY}).
%
\\
Finally, the following further definition will be useful in Section \ref{sec:low}.
\begin{definition}
	Let $M$ be a smooth compact Riemannian manifold embedded in $\mathbb{R}^N$. The {\em radius of topological invariance} for $M$ is
	\[
	r(M) := \sup\{ \rho>0:\operatorname{cat}(M_\rho)=\operatorname{cat}(M)\}
	\]
	where $M_\rho:=\{x\in \mathbb{R}^N: d(x,M)<\rho\}$.
\end{definition}

\section{Functional setting}\label{sec:funct}
Using Lemma \ref{lem:BP} and Lemma \ref{lem:BP3} we have that positive solutions of \eqref{eq:BPP} are critical points of
the $C^1$ functional
\[
J_{\varepsilon}
:=
u\in H^{1}(M)
\mapsto
\frac{1}{2}\|u\|_\varepsilon^2
+\frac{1}{4\varepsilon^3}\int_M \phi_{u}u^{2} d\mu_g
-\frac{1}{p}|u^{+}|_{p,\varepsilon}^{p}\in\mathbb{R}.
\]
Moreover, let us consider the Nehari manifold
\[
\mathcal{N}_{\varepsilon}
:=\{ u\in H^{1}(M)\setminus\left\{ 0\right\}  : N_\varepsilon(u)=0\},
\]
where
\[
N_\varepsilon(u)
:=J_{\varepsilon}'(u)[u]
=\|u\|_{\varepsilon}^{2}+\frac{1}{\varepsilon^{3}}\int_M\phi_{u}u^{2}d\mu_{g}-|u^{+}|_{p,\varepsilon}^{p}.
\]
We have
\begin{lemma}\label{lemNehari}
If $p> 4$, then:
\begin{enumerate}[label=(\roman{*}), ref=\roman{*}]
	\item \label{tu}for every $u\in H^1(M)$ with $u^+\not\equiv 0$, there exists $t_u>0$ such that $t_u u\in \mathcal{N}_{\varepsilon}$;
	\item \label{tucont} the map $u\in H^1(M)\mapsto t_u\in(0,+\infty)$ is continuous;
\item \label{disc0}there exists $C>0$ such that, for every $\varepsilon>0$ and $u\in\mathcal{N}_{\varepsilon}$, $|u^{+}|_{p,\varepsilon}\geq C$;
\item \label{N'n0}for every $u\in\mathcal{N}_{\varepsilon}$, $N'_\varepsilon(u)\neq 0$;
\item \label{mepsilon} there exists $C>0$ such that,  for every $\varepsilon>0$, $m_{\varepsilon}:=\displaystyle{\inf_{\mathcal{N}_{\varepsilon}}J_{\varepsilon}}\geq C>0$.
\end{enumerate}
\end{lemma}
\begin{proof}
Let $u\in H^{1}(M)$ with $u^+\not\equiv 0$. Property (\ref{tu}) is an easy consequence of the fact that, using Lemma \ref{lem:BP2}, the function
	\begin{equation}\label{eq:Nt}
	\varphi_u(t)
	:=
	N_\varepsilon(tu)
	=
	t^2 \|u\|_{\varepsilon}^{2}
	+ \frac{t^4}{\varepsilon^{3}}\int_M\phi_{u}u^{2} d\mu_g
	-t^p|u^{+}|_{p,\varepsilon}^{p},
	\quad t>0,
\end{equation}
admits a unique zero $t_u$, due to the assumption $p>4$.\\
To get claim \eqref{tucont} we observe that, if $u_n\rightarrow u$ in $H^1(M)$, then $|u^+_n|_p\rightarrow |u^+|_p$ and, by Lemma \ref{lem:BP3},
$$\int_M\phi_{u_n}u_n^{2} d\mu_g \rightarrow  \int_M\phi_{u}u^{2} d\mu_g.$$
Then $\{t_{u_n}\}$ is bounded and so  it converges, up to subsequence, to some $\bar{t}\in\mathbb{R}$. At this point, since we are dealing with the unique solutions of $\varphi_{u_n}(t)=0$, it is easy to see that  $\bar{t}=t_u$.\\
Moreover, if $u\in\mathcal{N}_{\varepsilon}$, by Lemma \ref{lem:BP} and \eqref{imb}
\[
|u^{+}|_{p,\varepsilon}^{p}
=
\|u\|_{\varepsilon}^{2}+\frac{1}{\varepsilon^{3}}\int_M \phi_{u}u^{2} d\mu_g
\geq
\|u\|_{\varepsilon}^{2}
\geq
C |u|_{p,\varepsilon}^{2}
\geq
C |u^{+}|_{p,\varepsilon}^{2}
\]
and so we get (\ref{disc0}).\\
To prove (\ref{N'n0}) we observe that, if $u\in\mathcal{N}_{\varepsilon}$, then, by Lemma \ref{lem:BP3},
\[
N_\varepsilon'(u)[u]
=N_\varepsilon'(u)[u]-4N_\varepsilon(u)
=-2\|u\|_{\varepsilon}^{2}-(p-4)|u^{+}|_{p,\varepsilon}^{p}
\leq -C
<0
\]
being $p>4$.\\
Finally, observe that, if $u \in \mathcal{N}_{\varepsilon}$, since $p>4$, by \eqref{imb}, (\ref{bLem:BP}) of Lemma \ref{lem:BP}, and (\ref{disc0}), we have
\[
J_\varepsilon(u)
=\Big(\frac{1}{2}-\frac{1}{p}\Big) \|u\|_\varepsilon^2
+\Big(\frac{1}{4}-\frac{1}{p}\Big) \frac{1}{\varepsilon^{3}}\int_M\phi_{u}u^{2}d\mu_g
\geq
C>0.
\]
\end{proof}

\begin{lemma}
\label{lem:PSnehari}
Let $\varepsilon>0$ be fixed. Then, for every $c>0$, the functional $J_{\varepsilon}$ satisfies the Palais Smale condition at level $c$. Moreover, if $\left\{ u_{n}\right\}$ is a Palais Smale sequence for $J_{\varepsilon}|_{\mathcal{N}_{\varepsilon}}$, then it is a Palais Smale sequence also for $J_{\varepsilon}$.
\end{lemma}

\begin{proof}
Let $\varepsilon>0$ be fixed and $\left\{ u_{n}\right\}\subset H^1(M)$ be a Palais Smale sequence for $J_{\varepsilon}$ at level $c$, with $c>0$. Then
\[
c+o_n(1)+o_n(1)\|u_n\|_\varepsilon
\geq
J_{\varepsilon}(u_{n})-\frac{1}{4}J_{\varepsilon}'(u_{n})[u_{n}]
=
\frac{1}{4}\|u_{n}\|_{\varepsilon}^{2}+\left(\frac{1}{4}-\frac{1}{p}\right)|u_{n}^{+}|_{p,\varepsilon}^{p}
\geq
\frac{1}{4}\|u_{n}\|_{\varepsilon}^{2}
\]
and so $\{u_n\}$ is bounded in $H^1(M)$.
Thus, up to a subsequence, it converges weakly to a function $\bar{u}\in H^1(M)$ and $u_{n}\to\bar{u}$ in $L^{\tau}(M)$ for $1\le \tau<6$.\\
The remaining part is standard, since, if
$$R_\varepsilon := -\frac{1}{\varepsilon}\Delta_g +\frac{1}{\varepsilon^3} \operatorname{Id} $$
is the Riesz isomorphism on $H^1(M)$, we have
\[
u_n=-\frac{q^2}{\varepsilon^3}R_\varepsilon^{-1}(\phi_{u_n}u_n)
+\frac{1}{\varepsilon^3}R_\varepsilon^{-1}(|u_n^+|^{p-2}u_n^+)+o_n(1)
\]
and so we can conclude observing that $\{\phi_{u_n}u_n\}$ is bounded in $L^{3/2}(M)$, being
\[
|\phi_{u_n}u_n|_{3/2} \leq |\phi_{u_n}|_6 |u_n|_2,
\]
$\{|u_n^+|^{p-2}u_n^+\}$ is bounded in $L^{p'}(M)$, being $p'$ the conjugate exponent of $p$, and using the compact embedding (by duality) of $L^{\tau'}(M)$ in $H^{-1}(M)$.
Thus, up to a subsequence $u_n\to\bar{u}$ in $H^1(M)$.\\
Let now $\left\{ u_{n}\right\}\subset\mathcal{N}_{\varepsilon}$ be a Palais Smale sequence of $J_{\varepsilon}|_{\mathcal{N}_{\varepsilon}}$, namely such that $J_{\varepsilon}(u_{n})\rightarrow c$ and $J_{\varepsilon}'(u_{n})-\lambda_n N_\varepsilon'(u_n)\rightarrow0$, with $\{\lambda_n\}\subset \mathbb{R}$.\\
To prove that $J_{\varepsilon}'(u_{n})\rightarrow0$, first we observe that, arguing as before, we have
that $\{u_n\}$ is bounded in $H^1(M)$.\\
Moreover, arguing as in the proof of (\ref{N'n0}) of Lemma \ref{lemNehari} and using (\ref{disc0}) of Lemma \ref{lemNehari} we have
\[
N_\varepsilon'(u_n)[u_n]
=-2\|u_n\|_{\varepsilon}^{2}-(p-4)|u_n^{+}|_{p,\varepsilon}^{p}
<-C<0.
\]
Thus, up to a subsequence, $N_\varepsilon'(u_n)[u_n]\to \ell\in[-\infty,0)$.\\
Then, since $\left\{ u_{n}\right\}\subset\mathcal{N}_{\varepsilon}$, we have that $\lambda_n N_\varepsilon'(u_n)[u_n]\to0$ and so $\lambda_n \to 0$.
Hence we can conclude proving that $\{N_\varepsilon'(u_n)\}$ is bounded in $H^{-1}(M)$ and this is an immediate consequence of the H\"older inequality and of the boundedness of $\{u_n\}$ in $H^1(M)$, since, for every $\varphi\in H^1(M)$,
\begin{align*}
|N_\varepsilon'(u_n)[\varphi]|
&\leq
2 |\langle u_n,\varphi \rangle_\varepsilon|
+\frac{4}{\varepsilon^3}\int_{M} |\phi_{u_n} u_n \varphi| d\mu_g
+\frac{p}{\varepsilon^3} \int_{M} |u_n^+|^{p-1}|\varphi| d\mu_g\\
&\leq
C[\|u_n\|_\varepsilon + \|u_n\|_\varepsilon^3+ \|u_n\|_\varepsilon^{p-1}  ]
\|\varphi\|_\varepsilon\\
&\leq C \|\varphi\|_\varepsilon.
\end{align*}
\end{proof}

\section{Low energy solutions}\label{sec:low}

We start now the proof of Theorem \ref{th:main}. In particular, this section is devoted to show the existence of multiple low energy solutions by the {\em photography method}. 
The core of the proof relies in three claims: Lemma \ref{lem:Psieps} where we show that, for small $\varepsilon>0$, the function $\Psi_\varepsilon$, defined in \eqref{Psieps}, maps points of $M$ in low energy functions in $\mathcal{N}_\varepsilon$, Lemma \ref{lem:gamma} that prevents vanishing for low energy functions, Propositon \ref{prop:baricentro} which states that  low energy functions in $\mathcal{N}_\varepsilon$ are indeed concentrated around a point on the manifold. In light of these results, Propositon \ref{centrodimassa}  establishes the link between the points on $M$ and the set of low energy functions, which allows us to apply classical result Theorem \ref{prelim} and to get the first claim of Theorem \ref{th:main}.
	
Firstly, we give a good model for low energy  solutions of \eqref{eq:BPP}.

Let $U\in H^{1}(\mathbb{R}^{3})$ be the unique positive solution of 
\begin{equation}
-\Delta u+ u=|u|^{p-2}u\text{ in }\mathbb{R}^{3}.\label{eq:problim}
\end{equation}
It is well known that such a function is radially symmetric, nondegenerate, and decays exponentially at infinity (see \cite{GNN,K}).
For $\xi\in M$ and $\varepsilon>0$,  let us take
\begin{equation}\label{Wxieps}
W_{\xi,\varepsilon}:=U_{\varepsilon}(\exp_{\xi}^{-1}\cdot)\chi_r(|\exp_{\xi}^{-1}\cdot|)
\end{equation}
where $\chi$ is a cut off such that
\begin{equation*}
\chi_r(\rho)
:=
\begin{cases}
1&\text{if } \rho\in[0,r/2),\\
0&\text{if } \rho\in (r,+\infty),
\end{cases}
\quad
|\chi'|\leq 2/r,
\quad
|\chi''|\leq 2/r^2,
\end{equation*}
$r>0$ being the injectivity radius defined in Section \ref{sec:prelim}, and $U_\varepsilon=U(\cdot/\varepsilon)$.

Let us prove the following preliminary result.
\begin{lemma}\label{lem:limiti}
We have:
\begin{enumerate}[label=(\roman{*}), ref=\roman{*}]
\item \label{1lem:limiti}${\displaystyle \lim_{\varepsilon\rightarrow0}\|W_{\xi,\varepsilon}\|_{\varepsilon}^{2}
=\|U\|_{H^1}^2}$;
\item \label{2lem:limiti}${\displaystyle \lim_{\varepsilon\rightarrow0}\left|W_{\xi,\varepsilon}\right|_{q,\varepsilon}^{q}=|U|_{q}^{q}}$ for $q\in [1,6]$;
\item \label{2blem:limiti}${\displaystyle \lim_{\varepsilon\rightarrow0}
	\frac{1}{\varepsilon^{3}}\int_{M}W_{\xi,\varepsilon}^{2}\phi_{W_{\xi,\varepsilon}} = 0}$;
\item \label{3lem:limiti}${\displaystyle \lim_{\varepsilon\rightarrow0}t_{W_{\xi,\varepsilon}}=1}$.
\end{enumerate}
\end{lemma}
\begin{proof}
Let $\tilde{W}_{\xi,\varepsilon}:=W_{\xi,\varepsilon}(\exp_\xi \cdot)$ and $ |g_\xi|:= \det (g_{ij})$, where $(g_{ij})$  is the inverse matrix of 
$(g_\xi^{ij})$.
Then $\tilde{W}_{\xi,\varepsilon}=U(\cdot/\varepsilon)\chi(|\cdot|)$ and we have
\begin{align*}
\frac{1}{\varepsilon}\int_M |\nabla_g W_{\xi,\varepsilon}|^2d\mu_g
&=
\frac{1}{\varepsilon}
\int_{B(0,r)}
g_\xi^{ij}(y)
\partial_i \tilde{W}_{\xi,\varepsilon}(y)
\partial_j \tilde{W}_{\xi,\varepsilon}(y)
|g_\xi(y)|^{1/2} dy\\
&=
\int_{B(0,r/\varepsilon)}
\Big(
g_\xi^{ij}(\varepsilon z) 
\partial_i U(z) 
\partial_j U (z)
\Big)
\chi^2(|\varepsilon z|) |g_\xi(\varepsilon z)|^{1/2} dz\\
&\quad
+\varepsilon \int_{B(0,r/\varepsilon)}
\Big(
g_\xi^{ij}(\varepsilon z) 
\partial_i U(z) 
\frac{y_j}{|z|}
\Big)
U(z) \chi(|\varepsilon z|)\chi'(|\varepsilon z|)
|g_\xi(\varepsilon z)|^{1/2} dz\\
&\quad
+\varepsilon \int_{B(0,r/\varepsilon)}
\Big(
g_\xi^{ij}(\varepsilon z) 
\frac{y_i}{|z|}
\partial U_j(z) 
\Big)U(z)\chi(|\varepsilon z|)\chi'(|\varepsilon z|)
|g_\xi(\varepsilon z)|^{1/2} dz\\
&\quad
+\varepsilon^2\int_{B(0,r/\varepsilon)}
\Big(
g_\xi^{ij}(\varepsilon z) 
z_i z_j\Big)
U^2(z) \Big(\frac{\chi'(|\varepsilon z|)}{|z|}\Big)^2
|g_\xi(\varepsilon z)|^{1/2} dz
\end{align*}
and
\[
	\int_M | W_{\xi,\varepsilon}|^2 d\mu_g
	=
	\int_{B(0,r)}
	|U_\varepsilon(y)\chi(|x|)|^2 |g_\xi(y)|^{1/2} dy
	=
	\varepsilon^3 \int_{B(0,r/\varepsilon)}
	|U(z)\chi(|\varepsilon z|)|^2 |g_\xi(\varepsilon z)|^{1/2} dz.
\]
Applying the Dominated Convergence Theorem and using \eqref{asymptg} we get
\[
\frac{1}{\varepsilon}\int_M |\nabla_g W_{\xi,\varepsilon}|^2d\mu_g
+\frac{1}{\varepsilon^3} \int_M | W_{\xi,\varepsilon}|^2 d\mu_g
\to
|\nabla U|_{2}^{2}+|U|_{2}^{2}.
\]
Analogously
\[
\frac{1}{\varepsilon^3}\int_M | W_{\xi,\varepsilon}|^p d\mu_g
=
\frac{1}{\varepsilon^3}\int_{B(0,r)}
|U_\varepsilon(y)\chi(|y|)|^p |g_\xi(y)|^{1/2} dy
=
\int_{B(0,r/\varepsilon)}
|U(z)\chi(|\varepsilon z|)|^p |g_\xi(\varepsilon z)|^{1/2} dz
\to
|U|_{p}^{p}.
\]
To prove (\ref{2blem:limiti}), observe that, by (\ref{tu}) in Lemma \ref{lemNehari}
and Lemma \ref{lem:BP2}, $t_{W_{\xi,\varepsilon}}$ satisfies
\begin{equation}\label{eq:HprimoW}
	t_{W_{\xi,\varepsilon}}^{p-2}\left|W_{\xi,\varepsilon}\right|_{p,\varepsilon}^{p}
	=\|W_{\xi,\varepsilon}\|_{\varepsilon}^{2}
	+t_{W_{\xi,\varepsilon}}^2\frac{1}{\varepsilon^{3}}\int_{M}W_{\xi,\varepsilon}^{2}\phi_{W_{\xi,\varepsilon}} d\mu_{g}.
\end{equation}
By (\ref{aLem:BP}) in Lemma \ref{lem:BP} we have
\begin{equation*}
	\int_{M}W_{\xi,\varepsilon}^{2}\phi_{W_{\xi,\varepsilon}} d\mu_{g}
	=\| \phi_{W_{\xi,\varepsilon}}\|_{H^2}^2
	\leq
	C\left|W_{\xi,\varepsilon}\right|_{2}^{4}
\end{equation*}
and so,
\begin{equation}\label{eq:Geps}
	0
	\leq
	\frac{1}{\varepsilon^{3}}\int_{M}W_{\xi,\varepsilon}^{2}\phi_{tW_{\xi,\varepsilon}} d\mu_{g}
	\leq C \varepsilon^3 \left|W_{\xi,\varepsilon}\right|_{2,\varepsilon}^{4}
	\to 0.
\end{equation}
Finally, writing \eqref{eq:HprimoW} as
\[
t_{W_{\xi,\varepsilon}}^{2}
\Big(
t_{W_{\xi,\varepsilon}}^{p-4}\left|W_{\xi,\varepsilon}\right|_{p,\varepsilon}^{p}
-\frac{1}{\varepsilon^{3}}\int_{M}W_{\xi,\varepsilon}^{2}\phi_{W_{\xi,\varepsilon}} d\mu_{g}
\Big)
=\|W_{\xi,\varepsilon}\|_{\varepsilon}^{2}
\]
and using (\ref{1lem:limiti}), (\ref{2lem:limiti}), and \eqref{eq:Geps}, we get that $t_{W_{\xi,\varepsilon}}$ is bounded for $\varepsilon$ small enough.
Hence, again by \eqref{eq:HprimoW}, (\ref{1lem:limiti}), and (\ref{2lem:limiti}),
\[
t_{W_{\xi,\varepsilon}}^{p-2}
=\frac{1}{\left|W_{\xi,\varepsilon}\right|_{p,\varepsilon}^{p}}
\Big(\|W_{\xi,\varepsilon}\|_{\varepsilon}^{2}
+t_{W_{\xi,\varepsilon}}^2\frac{1}{\varepsilon^{3}}\int_{M}W_{\xi,\varepsilon}^{2}\phi_{W_{\xi,\varepsilon}} d\mu_{g} \Big)
\to
\frac{\| U\|_{H^1}^{2}}{|U|_{p}^{p}}=1.
\]
\end{proof}
Now let us define, for every fixed $\varepsilon>0$, the continuous map
\begin{equation}\label{Psieps}
\Psi_{\varepsilon}:= \xi\in M
\mapsto
t_{W_{\xi,\varepsilon}} W_{\xi,\varepsilon}\in \mathcal{N}_{\varepsilon}
\end{equation}
and
\begin{equation}\label{minfty}
m_{\infty}
:=
\frac{p-2}{2p}|U|_{p}^{p}
=
\frac{p-2}{2p}\|U\|_{H^1}^{2},
\end{equation}
which corresponds to the energy level of $U$ with respect to equation \eqref{eq:problim}.

Using Lemma \ref{lem:limiti} we have
\begin{lemma}\label{lem:Psieps}
For every $\delta>0$ there exists $\varepsilon_{0}=\varepsilon_{0}(\delta)>0$ such that, for every $\varepsilon\in (0,\varepsilon_{0})$, and for every $\xi\in M$, $J_{\varepsilon}(\Psi_{\varepsilon}(\xi))<m_{\infty}+\delta$.
\end{lemma}

As an immediate consequence of Lemma \ref{lem:Psieps}, using Lemma \ref{lemNehari}, we have that there exists $C>0$ such that
\begin{equation}\label{rem:limsup}
C
\leq \liminf_{\varepsilon\rightarrow0}m_{\varepsilon}
\leq \limsup_{\varepsilon\rightarrow0}m_{\varepsilon}
\leq m_{\infty}.
\end{equation}

Now, as in \cite{BBM}, we need to consider a {\em good} partition of the manifold $M$ which we define as follows.
\begin{definition}\label{defgoodp}
For a given $\varepsilon>0$ we say that a finite partition ${\mathcal P}_{\varepsilon}=\{ P_{j}^{\varepsilon}\} _{j\in\Lambda_{\varepsilon}}$
of the manifold $M$ is a {\em good} partition of $M$ if:
\begin{enumerate}[label=(\arabic{*}), ref=\arabic{*}]
	\item for any $j\in\Lambda_{\varepsilon}$ the set $P_{j}^{\varepsilon}$
	is closed;
	\item  $P_{i}^{\varepsilon}\cap P_{j}^{\varepsilon}\subset\partial P_{i}^{\varepsilon}\cap\partial P_{j}^{\varepsilon}$ for any $i\ne j$;
	\item \label{def3} there exist $r_{1}(\varepsilon)>0$ 
	such that there are points $q_{j}^{\varepsilon}\in P_{j}^{\varepsilon}$
	for which  $B_{g}(q_{j}^{\varepsilon},\varepsilon)\subset P_{j}^{\varepsilon}\subset 
	B_{g}(q_{j}^{\varepsilon},r_{1}(\varepsilon))$,
	with $r_{1}(\varepsilon)\ge
	C\varepsilon$ for
	some positive constant $C$;
	\item \label{def4}there exists $\nu(M)\in\mathbb{N}$, independent of $\varepsilon$, such that every $\xi\in M$ is contained in
	at most $\nu(M)$ balls $B_{g}(q_{j}^{\varepsilon},r_{1}(\varepsilon))$. 
\end{enumerate}
\end{definition}

The existence of {\em good} partitions easily follows observing that, for $\varepsilon$ small enough, condition (\ref{def4}) in Definition \ref{defgoodp} can be satisfied by the compactness of $M$. Thus, without loss of generality, we can assume that, given $\delta>0$,  $\varepsilon_{0}(\delta)$ in Lemma \ref{lem:Psieps} is sufficiently small to ensure also the existence of a {\em good} partition for every $\varepsilon\in(0,\varepsilon_0(\delta))$.

Thus, arguing  as in \cite[Lemma 5.3]{BBM}, we get the following result which prevents vanishing on the Nehari manifold.

\begin{lemma}\label{lem:gamma}
There exists a constant $\gamma>0$ such that for any $\delta>0$ and for any $\varepsilon\in(0,\varepsilon_{0}(\delta))$, given any {\em good} partition ${\mathcal P}_{\varepsilon}=\{ P_{j}^{\varepsilon}\} _{j}$ of the manifold $M$ and for any function $u\in{\mathcal N}_{\varepsilon}$, there exists $\bar{j}\in \Lambda_\varepsilon$ such that 
\[
\frac{1}{\varepsilon^{3}}\int_{P_{\bar{j}}^{\varepsilon}}|u^{+}|^{p}d\mu_g\ge\gamma.
\]
\end{lemma}

\begin{proof}
Observe that, if $u \in \mathcal{N}_{\varepsilon}$, by (\ref{bLem:BP}) of Lemma \ref{lem:BP},
\begin{equation}\label{conticini}
\begin{split}
\|u\|_{\varepsilon}^{2}
& =
\frac{1}{\varepsilon^{3}} |u^{+}|_{p}^{p}
-\frac{q^2}{\varepsilon^{3}}\int_{M}\phi_u u^{2} d\mu_{g} \le\frac{1}{\varepsilon^{3}}|u^{+}|_{p}^{p}
=
\sum_{j\in \Lambda_\varepsilon}\frac{1}{\varepsilon^{3}}\int_{P_{j}^\varepsilon}|u^{+}|^{p}d\mu_{g}\\
&=
\sum_{j\in \Lambda_\varepsilon} \Big( \frac{1}{\varepsilon^{3}} \int_{P_{j}^\varepsilon}|u^{+}|^{p}d\mu_{g}  \Big)^{1-\frac{2}{p}}
\Big( \frac{1}{\varepsilon^{3}}\int_{P_{j}^\varepsilon}|u^{+}|^{p}d\mu_{g}  \Big)^{\frac{2}{p}}\\
&\leq
\max_{j\in \Lambda_\varepsilon}\left\{\Big( \frac{1}{\varepsilon^{3}} \int_{P_{j}^\varepsilon}|u^{+}|^{p}d\mu_{g}  \Big)^{1-\frac{2}{p}}\right\}
\sum_{j\in \Lambda_\varepsilon} \Big( \frac{1}{\varepsilon^{3}}\int_{P_{j}^\varepsilon}|u^{+}|^{p}d\mu_{g}  \Big)^{\frac{2}{p}}.
\end{split}
\end{equation}
Let now $\chi_\varepsilon\in C^\infty([0,+\infty[,[0,1])$ such that
\[
\chi_\varepsilon(t):=
\begin{cases}
	1&\text{if } t\leq r_2(\varepsilon)\\
	0&\text{if } t>r_1(\varepsilon)
\end{cases}
,\quad
|\chi_\varepsilon'|\leq \frac{K}{\varepsilon}, \text{for } K>0,
\]
and, for every $j\in \Lambda_\varepsilon$, $u_j:=u^+ \chi_\varepsilon(|\cdot-q_j^\varepsilon|)$, where $r_i(\varepsilon)$'s and $q_j^\varepsilon$ come from Definition \ref{defgoodp}.\\
For every $j\in \Lambda_\varepsilon$ we have that $u_j\in H^1(M)$ and
\begin{align*}
\Big( \frac{1}{\varepsilon^{3}}\int_{P_{j}^\varepsilon}|u^{+}|^{p}d\mu_{g}  \Big)^{\frac{2}{p}}
&\leq
|u_j|_{p,\varepsilon}^2
\leq
C\|u_j\|_\varepsilon^2
=
C (\| u^+|_{P_j^\varepsilon}\|_\varepsilon^2
+ \| {u_j}|_{B_g(q_j^\varepsilon,r_1(\varepsilon))\setminus P_j^\varepsilon}\|_\varepsilon^2)\\
&\leq
C\Big(
\| u^+|_{P_j^\varepsilon}\|_\varepsilon^2
+\frac{1}{\varepsilon}\int_{B_g(q_j^\varepsilon,r_1(\varepsilon))\setminus P_j^\varepsilon} |\nabla_g u^+|^2 d\mu_g
+\frac{K^2+\omega}{\varepsilon^3}\int_{B_g(q_j^\varepsilon,r_1(\varepsilon))\setminus P_j^\varepsilon} |u^+|^2 d\mu_g 
\Big)\\
&\leq
C\Big(
\| u^+|_{P_j^\varepsilon}\|_\varepsilon^2
+ \frac{K^2+\omega}{\omega} \| u^+|_{B_g(q_j^\varepsilon,r_1(\varepsilon))\setminus P_j^\varepsilon}\|_\varepsilon^2
\Big).
\end{align*}
Thus, using (\ref{def4}) in Definition \ref{defgoodp},
\begin{equation}
\label{conticinifin}
\sum_{j\in \Lambda_\varepsilon}\Big( \frac{1}{\varepsilon^{3}}\int_{P_{j}^\varepsilon}|u^{+}|^{p}d\mu_{g}  \Big)^{\frac{2}{p}}
\leq
C\frac{K^2+2\omega}{\omega} \nu(M) \| u^+\|_\varepsilon^2
\leq
C\frac{K^2+2\omega}{\omega} \nu(M) \| u\|_\varepsilon^2.
\end{equation}
Hence, by \eqref{conticini} and \eqref{conticinifin},
\[
\max_{j\in \Lambda_\varepsilon}\left\{\Big( \frac{1}{\varepsilon^{3}} \int_{P_{j}^\varepsilon}|u^{+}|^{p}d\mu_{g}  \Big)^{1-\frac{2}{p}}\right\}
\geq C
\]
and we conclude.
\end{proof}


Now, we can refine previuos result, obtaining concentration of low energy functions on the Nehari manifold. This is a key tool to get the multiplicity claim in Theorem \ref{th:main}.
\begin{proposition}
\label{prop:baricentro}For any $\eta\in(0,1)$ there exists $\delta_{0}<m_{\infty}$
such that for any $\delta\in(0,\delta_{0})$, for any $\varepsilon\in(0,\varepsilon_{0}(\delta))$, with $\varepsilon_{0}(\delta)$ as in Lemma \ref{lem:Psieps}, and for any function $u\in{\mathcal N}_{\varepsilon}\cap J_{\varepsilon}^{m_{\infty}+\delta}$
we can find a point $q=q(u)\in M$ such that
\[
\frac{1}{\varepsilon^{3}}\int_{B_g(q,r(M)/2)}|u^{+}|^{p} d\mu_g >\left(1-\eta\right)\frac{2p}{p-2}m_{\infty}.
\]
\end{proposition}

\begin{proof}
First let us show our thesis for $u\in{\mathcal N}_{\varepsilon}\cap J_{\varepsilon}^{m_{\varepsilon}+2\delta}$.\\
Assume by contradiction that there exists $\eta\in(0,1)$ and sequences $\{\delta_{k}\},\{\varepsilon_{k}\}\subset (0,+\infty)$, $\{ u_{k}\}\subset \mathcal{N}_{\varepsilon_{k}}$, such that $\delta_{k},\varepsilon_{k}\rightarrow0$ as $k\to+\infty$,
\begin{equation}\label{eq:mepsk}
m_{\varepsilon_{k}}
\leq
J_{\varepsilon_{k}}(u_{k})
\leq m_{\varepsilon_{k}}+2\delta_{k}
\end{equation}
and, for every $q\in M$, 
\begin{equation}\label{eq:contradiction}
\frac{1}{\varepsilon_{k}^{3}}\int_{B_{g}(q,r(M)/2)}|u_{k}^{+}|^{p}d\mu_{g}\le\left(1-\eta\right)\frac{2p}{p-2}m_{\infty}.
\end{equation}
First observe that $\{\|u_k\|_{\varepsilon_{k}}\}$ is bounded.
Indeed, by \eqref{rem:limsup}, \eqref{eq:mepsk}, for $k$ large enough,
\begin{align*}
2m_\infty
&\geq
m_{\varepsilon_{k}} +2 \delta_{k}
\geq
J_{\varepsilon_{k}} (u_k)
- \frac{1}{p} J_{\varepsilon_{k}}'(u_k)[u_k]
=
\Big(\frac{1}{2}-\frac{1}{p}\Big) \|u_k\|_{\varepsilon_{k}}^2
+ \Big(\frac{1}{4}-\frac{1}{p}\Big) \frac{1}{\varepsilon_{k}^3} \int_M \phi_{u_k} u_k^2 d\mu_g\\
&\geq
\Big(\frac{1}{2}-\frac{1}{p}\Big) \|u_k\|_{\varepsilon_{k}}^2.
\end{align*}
Applying the Ekeland Principle (see \cite{deFig}) as in \cite[Lemma 5.4]{BBM}\footnote{For completeness, in Appendix \ref{EVP}  we give some  details.}, we get that, for every $\varphi\in H^1(M)$,
\begin{equation}\label{eq:ps}
	\left|J'_{\varepsilon_{k}}(u_{k})[\varphi]\right|
	\le
	C\sqrt{\delta_{k}}\|\varphi\|_{\varepsilon_{k}}.
\end{equation}
Moreover, by Lemma \ref{lem:gamma}, there exists a constant $\gamma>0$ such that for any $\delta>0$ and for any $k$ large enough, given a {\em good} partition ${\mathcal P}_{\varepsilon_k}=\{ P_{j}^{\varepsilon_k}\}_{j\in\Lambda_{\varepsilon_k}}$ of the manifold $M$, there exists $\bar{j}_k\in \Lambda_{\varepsilon_k}$ such that 
\[
	\frac{1}{\varepsilon_k^{3}}\int_{P_{\bar{j}_k}^{\varepsilon_k}}| u_k^{+}|^{p}d\mu_g\ge\gamma.
\]
Let $q_{k}\in P_{\bar{j}_k}^{\varepsilon_k}$ as in (\ref{def3}) of Definition \ref{defgoodp} and
\begin{equation}
\label{wk}
w_{k}(z):=u_{k}(\exp_{q_{k}}(\varepsilon_{k}z))\chi_{r}(\varepsilon_{k}|z|),
\end{equation}
where
$r$ is the injectivity radius.\\
Moreover, let $\tilde{u}_k(z)=u_k(\exp_{q_{k}}(z))$, $v_k(z)=\tilde{u}_k(\varepsilon_kz)$, and $\chi_k(|z|)=\chi_r(\varepsilon_k |z|)$, so that
\[
w_k(z)=\tilde{u}_k(\varepsilon_k z)\chi(\varepsilon_k |z|)=v_k(z)\chi_k(|z|).
\]
We have that $w_{k}\in H_{0}^{1}(B(0,r/\varepsilon_{k}))\subset H^{1}(\mathbb{R}^{3})$, and
\begin{align*}
\| w_k\|_{H^1}^2
&=
\int_{B(0,\frac{r}{\varepsilon_k})}  \Big|\chi_k(|z|)\nabla v_k(z)+\varepsilon_k\chi_r'(\varepsilon_k|z|)\frac{z}{|z|}v_k(z) \Big|^2 dz
+\int_{B(0,\frac{r}{\varepsilon_k})}  |\chi_k(|z|) v_k(z)|^2 dz
\\
&\le
2\int_{B(0,\frac{r}{\varepsilon_k})}  |\nabla v_k(z)|^2dz
+\frac{8}{r^2}\varepsilon_k^2 \int_{B(0,\frac{r}{\varepsilon_k})} |v_k(z)|^2 dz
+ \int_{B(0,\frac{r}{\varepsilon_k})}  |v_k(z)|^2 dz
\\
&\le
2\varepsilon_k^2\int_{B(0,\frac{r}{\varepsilon_k})}  |\nabla \tilde{u}_k(\varepsilon_k z)|^2dz
+ C \int_{B(0,\frac{r}{\varepsilon_k})}  |\tilde{u}_k(\varepsilon_k z)|^2dz
\\
&\le
C\varepsilon_k^2 \int_{B(0,\frac{r}{\varepsilon_k})}  g_{q_k}^{ij}(\varepsilon_k z)\frac{\partial \tilde{u}_k}{\partial z_i}(\varepsilon_k z)\frac{\partial \tilde{u}_k}{\partial z_j}(\varepsilon_k z) |g_{q_k}(\varepsilon_k z)|^{1/2}dz
+C \int_{B(0,\frac{r}{\varepsilon_k})}  |\tilde{u}_k(\varepsilon_k z)|^2 |g_{q_k}(\varepsilon_k z)|^{1/2} dz
\\
&=
\frac{C}{\varepsilon_k}\int_{B(0,r)}  g_{q_k}^{ij}(y)\frac{\partial \tilde{u}_k}{\partial z_i}(y)\frac{\partial \tilde{u}_k}{\partial z_j}(y) |g_{q_k}(y)|^{1/2}dy
+\frac{C}{\varepsilon_k^3}\int_{B(0,r)}  |\tilde{u}_k(y)|^2 |g_{q_k}(y)|^{1/2} dy
\\
&\leq C\|u_{k}\|_{\varepsilon_{k}}^{2}\le C.
\end{align*}
Thus $w_{k}$ converges weakly in $H^{1}(\mathbb{R}^{3})$ and strongly in $L_{\text{loc}}^{t}(\mathbb{R}^{3})$, $t\in[1,6)$, to a function $w\in H^{1}(\mathbb{R}^{3})$.\\
Let us prove that $w\ge0$ and it solves weakly 
\begin{equation}\label{eq:weq}
-\Delta w+ w=w^{p-1}
\text{ in } \mathbb{R}^3.
\end{equation}
Let $\varphi\in C_{0}^{\infty}(\mathbb{R}^{3})$. There exists $k\in\mathbb{N}$ such that $\operatorname{spt}(\varphi)\subset B(0,r/2\varepsilon_{k})$. Define
\[
\varphi_k:M\to\mathbb{R},
\qquad
\varphi_{k}(x):=\varphi\left(\frac{1}{\varepsilon_{k}}\exp_{q_{k}}^{-1}(x)\right).
\]
We have that
$\operatorname{spt}(\varphi_{k})\subset B_{g}(q_{k},r/2)$ and, being
\begin{equation}
\label{phitildek}
\tilde{\varphi}_{k}(y):=\varphi_{k}(\exp_{q_{k}}(y))=\varphi(y/\varepsilon_k),
\end{equation}
then
\begin{align*}
\|\varphi_{k}\|_{\varepsilon_{k}}^{2}
&=
\frac{1}{\varepsilon_k}\int_{B(0,\frac{r}{2})}  g_{q_k}^{ij}(y)  \frac{\partial \tilde{\varphi}_{k}}{\partial y^i}(y) \frac{\partial \tilde{\varphi}_{k}}{\partial y^j}(y) |g_{q_k}(y)|^{1/2} dy
+ \frac{1}{\varepsilon_k^3} \int_{B(0,\frac{r}{2})} |\tilde{\varphi}_{k}(y)|^2 |g_{q_k}(y)|^{1/2} dy
\\
&=
\int_{B(0,\frac{r}{2\varepsilon_k})}  g_{q_k}^{ij}(\varepsilon_k z)  \frac{\partial \varphi}{\partial z^i}(z) \frac{\partial \varphi}{\partial z^j}(z) |g_{q_k}(\varepsilon_k z)|^{1/2} dz
+   \int_{B(0,\frac{r}{2\varepsilon_k})} |\varphi(z)|^2 |g_{q_k}(\varepsilon_k z)|^{1/2} dz
\\
&\leq
C\|\varphi\|_{H^1}^2.
\end{align*}
Thus, by \eqref{eq:ps}, 
\begin{equation}\label{eq:stella}
|J'_{\varepsilon_{k}}(u_{k})[\varphi_{k}]|
\le C \sqrt{\delta_k} \|\varphi_{k}\|_{\varepsilon_{k}}
\le C \sqrt{\delta_k} \|\varphi\|_{H^1}
\rightarrow0\text{ as }k\rightarrow+\infty.
\end{equation}
On the other hand, observe that
\begin{equation}
\label{convmisto}
\frac{1}{\varepsilon_{k}^3}\int_{M} \phi_{u_{k}} u_{k}  \varphi_{k}  d\mu_g \to 0
\text{ as }k\to+\infty
\end{equation}
since, by \eqref{phitildek},
\begin{align*}
\int_{M} | \varphi_{k}|^3 d\mu_g
&=
\int_{B_{g}(q_k,\frac{r}{2})} | \varphi_{k}|^3 d\mu_g
=
\int_{B(0,\frac{r}{2})}  | \tilde{\varphi}_{k} (y)|^3 |g_{q_{k}}(y)|^{1/2}  dy
\leq
C\varepsilon_{k}^3 \int_{B(0,\frac{r}{2\varepsilon_{k}})} | \varphi(z)|^3 dy
\leq
C\varepsilon_{k}^3 | \varphi|_3^3
\end{align*}
and so, by H\"older inequality and (\ref{aLem:BP}) in Lemma \ref{lem:BP},
\[
\left|\int_{M} \phi_{u_{k}} u_{k}  \varphi_{k} d\mu_g \right|
\leq
|\phi_{u_k}|_3 |u_k|_3 |\varphi_{u_k}|_3
\leq
C \varepsilon_k \| \phi_{u_{k}} \|_{H^2} |u_k|_3
\leq
C \varepsilon_k  |u_k|_2^2  |u_k|_3
=
C \varepsilon_k^5  |u_k|_{2,\varepsilon_{k}}^2  |u_k|_{3,\varepsilon_{k}}
\leq
C \varepsilon_k^5.
\]
Then, using $\sim$ to denote the composition of the funcion with $\exp_{q_{k}}$ as before, since
\[
w_k\Big(\frac{y}{\varepsilon_k}\Big)=\tilde{u}_k(y)\chi_r(y)=\tilde{u}_k(y)
\quad\text{in } B(0,r/2)
\]
by \eqref{wk} and \eqref{phitildek},
\begin{align*}
J'_{\varepsilon_k}(u_{k})[\varphi_{k}]
&=
\frac{1}{\varepsilon_{k}^{3}}
\left(
\varepsilon_{k}^{2}\int_{B_g(q_k,r/2)}\nabla_{g}u_{k}\nabla_{g}\varphi_{k} d\mu_{g}
+ \int_{B_g(q_k,r/2)}u_{k}\varphi_{k} d\mu_{g}
- \int_{B_g(q_k,r/2)} |u_{k}^{+}|^{p-2} u_{k}^{+} \varphi_{k} d\mu_{g}
\right)\\
&\qquad
+o_k(1)\\
&=
\frac{1}{\varepsilon_{k}^{3}}
\left(
\varepsilon_{k}^{2}\int_{B(0,\frac{r}{2})}  g_{q_{k}}^{ij}(y) \partial_i \tilde{u}_k (y) \partial_j \tilde{\varphi}_k (y) 
|g_{q_{k}}(y)|^{1/2} dy
+ \int_{B(0,\frac{r}{2})} \tilde{u}_k(y) \tilde{\varphi}_k(y) |g_{q_{k}}(y)|^{1/2}dy
\right.\\
&\qquad\qquad
\left.
- \int_{B(0,\frac{r}{2})} |\tilde{u}_{k}^{+}(y) |^{p-2}\tilde{u}_{k}^{+}(y) \tilde{\varphi}_{k}(y) |g_{q_{k}}(y)|^{1/2} dy
\right)
+o_k(1)\\
&=
\int_{T_k}  g_{q_{k}}^{ij}(\varepsilon_{k} z) \frac{\partial w_k}{\partial z_i} (z) \frac{\partial \varphi}{\partial z_j} (z) 
|g_{q_{k}}(\varepsilon_{k} z)|^{1/2} dz
+\int_{T_k} w_k(z) \varphi (z) |g_{q_{k}}(\varepsilon_{k} z)|^{1/2}dz
\\
&\qquad\qquad
- \int_{T_k} |w_{k}^{+}(z) |^{p-2} w_{k}^{+}(z) \varphi (z) |g_{q_{k}}(\varepsilon_{k} z)|^{1/2} dz +o_k(1),
\end{align*}
where $T_{k}:=B(0,r/2\varepsilon_{k})\cap\operatorname{spt}(\varphi)$ and, for $k$ large enough, $T_{k}\equiv\operatorname{spt}(\varphi)$.\\
Hence, since by \eqref{asymptg},
\[
\int_{T_k}  g_{q_{k}}^{ij}(\varepsilon_{k} z) \frac{\partial w_k}{\partial z_i} (z) \frac{\partial \varphi}{\partial z_j} (z) |g_{q_{k}}(\varepsilon_{k} z)|^{1/2} dz
=
\int_{T_k} \nabla w_k (z) \nabla \varphi (z)  dz + O(\varepsilon_k^2),
\]
\[
\int_{T_k} w_k(z) \varphi (z) |g_{q_{k}}(\varepsilon_{k} z)|^{1/2}dz
=\int_{T_k} w_k(z) \varphi (z) dz+ O(\varepsilon_k^2),
\]
and
\[
\int_{T_k} |w_{k}^{+}(z) |^{p-2} w_{k}^{+}(z) \varphi (z) |g_{q_{k}}(\varepsilon_{k} z)|^{1/2} dz
=\int_{T_k} |w_{k}^{+}(z) |^{p-2} w_{k}^{+}(z) \varphi (z) dz + O(\varepsilon_k^2),
\]
\eqref{eq:stella}, and the convergence properties of $\{w_k\}$ imply that $w\ge0$ solves \eqref{eq:weq}.\\
Let us show that $w\not\equiv0$.\\
Let $\mathcal{N}_{\infty}:=\left\{ v\in H^{1}(\mathbb{R}^{3})\setminus\left\{ 0\right\} : |\nabla w|_2^2 + \omega |w|_2^2=|w|_{p}^{p}\right\}$ and $T>0$ large enough such that $P_{j_k}^{\varepsilon_k}\subset B_{g}(q_{k},\varepsilon_{k}T)$.
By \eqref{asymptg} and Lemma \ref{lem:gamma}, if $k$ is large enough, we have
\begin{align*}
\int_{B(0,T)}\left(w_{k}^{+}\right)^{p}dz
&=
\frac{1}{\varepsilon_{k}^{3}}\int_{B(0,\varepsilon_{k}T)}\left(u_{k}^{+}(\exp_{q_{k}}(y))\right)^{p} dy
\geq
\frac{C}{\varepsilon_{k}^{3}}\int_{B(0,\varepsilon_{k}T)}\left(u_{k}^{+}(\exp_{q_{k}}(y))\right)^{p}|g_{q_{k}}(y)|^{1/2}dy \\
& \geq
\frac{C}{\varepsilon_{k}^{3}}\int_{P_{k}^{\varepsilon_{k}}}|u_{k}^{+}|^{p}d\mu_{g}\ge\gamma.
\end{align*}
Hence $w\neq0$, $w\in\mathcal{N}_{\infty}$, and so $w=U$ (see \eqref{eq:problim}) and, by \eqref{minfty},
\[
\|w\|_{H^1}^2 
=|w|_{p}^{p}
=\frac{2p}{p-2}m_{\infty}.
\]
Then, there exists $T>0$ such that
\[
\int_{B(0,T)}w^{p}dz>\left(1-\frac{\eta}{8}\right)\frac{2p}{p-2}m_{\infty}
\]
and, since $w_k\rightarrow w$ in $L^p_{\text{loc}}(\mathbb{R}^3)$, for $k$ large enough we get
\begin{equation}\label{contradictionleft}
\int_{B(0,T)}\left(w_{k}^{+}\right)^{p}dz>\left(1-\frac{\eta}{4}\right)\frac{2p}{p-2}m_{\infty}.
\end{equation}
On the other hand, by \eqref{asymptg}, if $\sigma\in (0,3\eta/(4-\eta))$, we have that for $k$ sufficiently large,
\[
|g_{q_k}(\varepsilon_kz)|^{1/2}>1-\sigma
\text{ on } B(0,T)
\]
and so, by \eqref{eq:contradiction}, we get
\begin{align*}
\int_{B(0,T)}\left(w_{k}^{+}\right)^{p}dz
&\leq
\frac1{1-\sigma}\int_{B(0,T)}\left(u_{k}^+(\exp_{q_{k}}(\varepsilon_{k}z))\right)^p |g_{q_k}(\varepsilon_kz)|^{1/2} dz\\
&=
\frac1{(1-\sigma)\varepsilon_k^3}\int_{B(0,\varepsilon_k T)}\left(u_{k}^+(\exp_{q_{k}}(y))\right)^p|g_{q_k}(y)|^{1/2}dy\\
&\leq
\frac1{(1-\sigma)\varepsilon_k^3}\int_{B(0,r(M)/2)}\left(u_{k}^+(\exp_{q_{k}}(y))\right)^p|g_{q_k}(y)|^{1/2}dy\\
&=
\frac1{(1-\sigma)\varepsilon_k^3}\int_{B_{g}(q_k,r(M)/2)}|u_{k}^{+}|^{p}d\mu_{g}\\
&\leq
\frac{1-\eta}{1-\sigma} \frac{2p}{p-2} m_\infty
<
\left(1-\frac{\eta}{4}\right)\frac{2p}{p-2}m_{\infty},
\end{align*}
reaching a contradiction with \eqref{contradictionleft}.\\
Hence, in particular, for any $\eta\in(0,1)$ there exists $\delta_{0}<m_{\infty}$
such that, if $\{\delta_k\}\subset(0,\delta_{0})$ and $\delta_{k}\to 0$ as $k\to + \infty$ and $\varepsilon_k\in(0,\varepsilon_{0}(\delta_k))$, with $\varepsilon_{0}(\delta_k)$ as in Lemma \ref{lem:Psieps} and $\varepsilon_{k}\to 0$ as $k\to + \infty$, for any function $u_k\in{\mathcal N}_{\varepsilon_k}\cap J_{\varepsilon_k}^{m_{\varepsilon_k}+2\delta_k}$ we can find a point $q_k=q_k(u_k)\in M$ such that
\[
\frac{1}{\varepsilon^{3}_k}\int_{B_g(q_k,r(M)/2)}|u_k^{+}|^{p} d\mu_g >\left(1-\eta\right)\frac{2p}{p-2}m_{\infty}.
\]
Then
\[
m_{\varepsilon_{k}} +2\delta_{k}
\geq
J_{\varepsilon_{k}}(u_{k})
= \left(\frac{1}{2}-\frac{1}{p}\right)|u_{k}^{+}|_{p,\varepsilon_{k}}^{p}-\frac{q^{2}}{4\varepsilon_{k}^{3}}\int_{M}u_{k}^{2}\phi_{u_k}d\mu_{g}
>\left(1-\eta\right)m_{\infty} - \frac{q^{2}}{4\varepsilon_{k}^{3}}\int_{M}u_{k}^{2}\phi_{u_k}d\mu_{g}.
\]
Observe that
\[
\frac{1}{\varepsilon_{k}^{3}}\int_{M}u_{k}^{2}\phi_{u_k}d\mu_{g}
\leq
\frac{C}{\varepsilon_{k}^{3}} \| \phi_{u_k}\|_{H^2}^2
\leq
\frac{C}{\varepsilon_{k}^{3}} |u_k|_{2}^4
= C \varepsilon_{k}^{3} |u_k|_{2,\varepsilon_{k}}^4
\leq C \varepsilon_{k}^{3}.
\]
Thus, by \eqref{rem:limsup}, 
$$\lim_{k}m_{\varepsilon_{k}}= m_{\infty}.$$
Hence, when $\varepsilon,\delta$ are small enough, ${\mathcal N}_{\varepsilon}\cap J_{\varepsilon}^{m_{\infty}+\delta}\subset{\mathcal N}_{\varepsilon}\cap J_{\varepsilon}^{m_{\varepsilon}+2\delta}$
and the general claim follows from the first part of the proof.
\end{proof}

We remind that we are assuming $M$ to be smoothly embedded in some $\mathbb{R}^N$, for $N$ sufficiently large. 
At this point, if $u\in H^1(M)$, with $u^+\neq 0$, it is possible to define its barycenter as follows
\[
\beta(u):=\frac{1}{|u^+|_p^p}\int_M x [u^+(x)]^p d\mu_g\in \mathbb{R}^N.
\]
By the barycenter map, it is possible to associate to any low energy function a unique point which lies in $M_{r(M)}$, a neighborhood of the manifold. This is the final step to link the manifold to the set of low energy functions.
\begin{proposition}\label{centrodimassa}
There exists $\delta_{0}\in(0,m_{\infty})$ such that for any $\delta\in(0,\delta_{0})$, $\varepsilon\in(0,\varepsilon(\delta_{0}))$, and $u\in{\mathcal N}_{\varepsilon}\cap J_{\varepsilon}^{m_{\infty}+\delta}$,  $\beta(u)\in M_{r(M)}$. Moreover the composition $\beta\circ\Psi_{\varepsilon}:M\rightarrow M_{r(M)}$ is s homotopic to the immersion $i:M\rightarrow M_{r(M)}$.
\end{proposition}

\begin{proof}
By Proposition \ref{prop:baricentro}, given $\eta\in(0,1)$ and $\varepsilon,\delta$ small enough, if $u\in{\mathcal N}_{\varepsilon}\cap J_{\varepsilon}^{m_{\infty}+\delta}$, there exists $q=q(u)\in M$ such that
\[
\frac{1}{\varepsilon^{3}}\int_{B(q,r(M)/2)}(u^{+})^{p} d\mu_g
>\left(1-\eta\right)\frac{2p}{p-2}m_{\infty}.
\]
Moreover
\begin{equation}\label{sublevmeno}
m_{\infty}+\delta
\geq
J_{\varepsilon}(u)
=\frac{p-2}{2p} |u^{+}|_{p,\varepsilon}^{p}
-\frac{q^{2}}{4\varepsilon^{3}}\int_{M}u^{2}\phi_u d\mu_{g}
\end{equation}
and
\begin{equation}\label{bddu}
m_{\infty}+\delta
\geq
J_{\varepsilon} (u) - \frac{1}{p} J_{\varepsilon}'(u)[u]
\geq
\Big(\frac{1}{2}-\frac{1}{p}\Big) \|u\|_{\varepsilon}^2.
\end{equation}
By (\ref{aLem:BP}) in Lemma \ref{lem:BP} and \eqref{bddu}
\[
\frac{1}{\varepsilon^{3}}\int_{M}\phi_u u^{2} d\mu_g
\leq
\frac{C}{\varepsilon^{3}}|u|_{2}^{4}
\le C\varepsilon^{3}\|u\|_{\varepsilon}^{4}\le C\varepsilon^{3}.
\]
Thus, if $\varepsilon(\delta_{0})$ is small enough, by \eqref{sublevmeno},
\[
\frac{p-2}{2p} |u^{+}|_{p,\varepsilon}^{p}
\leq m_{\infty}+2\delta_{0}
\]
and so
\[
\frac{1}{\varepsilon^{3}|u^{+}|_{p,\varepsilon}^{p}} \int_{B_g(q,r(M)/2)}(u^{+})^{p} d\mu_g
>\frac{(1-\eta)m_\infty}{m_\infty+2\delta_{0}}.
\]
Hence
\begin{align*}
|\beta(u)-q|
& \leq
\frac{1}{\varepsilon^3 |u^{+}|_{p,\varepsilon}^{p}}
\left|\int_{M}(x-q)(u^{+})^{p} d\mu_g\right|\\
& \le
\frac{1}{\varepsilon^3 |u^{+}|_{p,\varepsilon}^{p}}
\left(
\left|\int_{B_g(q,r(M)/2)}(x-q)(u^{+})^{p} d\mu_g\right|
+ \left|\int_{M\setminus B_g(q,r(M)/2)}(x-q)(u^{+})^{p} d\mu_g\right|
\right)
\\
& \le \frac{r(M)}{2}
+\operatorname{diam}(M)\left(1-\frac{(1-\eta)m_{\infty}}{m_{\infty}+2\delta_{0}}\right)
<r(M),
\end{align*}
for $\delta_{0}$ and $\eta$ small. Here $\operatorname{diam}(M)$ denotes the diameter of the manifold $M$ as subset of $\mathbb{R}^{N}$.\\
The second part is standard (see for instance \cite[Proposition 5.11]{BBM}).
\end{proof}

To conclude the proof of the first claim of Theorem \ref{th:main}, we recall a classical result in topological methods.
\begin{theorem}\label{prelim}
	Let $J$ be a $C^{1,1}$ real functional on a complete $C^{1,1}$
	manifold $\mathcal{N}$. If $J$ is bounded from below and satisfies
	the Palais Smale condition then has at least $\operatorname{cat}(J^{d})$ critical
	point in $J^{d}$ where $J^{d}=\{u\in\mathcal{N}\ :\ J(u)\leq d\}$. Moreover
	if $\mathcal{N}$ is contractible and $\operatorname{cat} J^{d}>1$, there exists
	at least one critical point $u\not\in J^{d}$
\end{theorem}

The proof of this theorem combines \cite[Theorem 9.10]{AM}, for the first part, and a consequence of \cite[Deformation Lemma 7.11]{AM} for the additional solution when $\operatorname{cat} J^{d}>1$, due to the fact that a change of topology generates a (PS) sequence.

To complete the proof of the multiplicity result, it is sufficient to apply the following
\begin{remark}\label{rem:cat}
	Let $X_{1}$ and $X_{2}$ be topological spaces. If $g_{1}:X_{1}\rightarrow X_{2}$
	and $g_{2}:X_{2}\rightarrow X_{1}$ are continuous operators such
	that $g_{2}\circ g_{1}$ is homotopic to the identity on $X_{1}$,
	then $\operatorname{cat} X_{1}\leq\operatorname{cat} X_{2}$. 
\end{remark}
Indeed, for $g_1=\Psi_{\varepsilon}$ and $g_2=\beta$, we obtain $\operatorname{cat} {\mathcal N}_{\varepsilon}\cap J_{\varepsilon}^{m_{\infty}+\delta}\ge \operatorname{cat} M$, concluding the proof of the first claim of Theorem \ref{th:main}.


\section{Profile description}\label{sec:prof}

In this section we give a qualitative and somehow quantitative description of the solutions found in the previous section.

Let $\delta>0$ and $u_{\varepsilon}$ be a nontrivial (positive) solution of \eqref{eq:BPPnl} such that $J_{\varepsilon}(u_{\varepsilon})=m_\varepsilon\le m_{\infty}+\delta<2m_{\infty}$.\\
First we observe that, since $\{\|u_\varepsilon\|_\varepsilon \}$ is bounded, then, by Lemma \ref{lem:BP},
\[
\|\phi_{u_\varepsilon}\|_{C^2(M)} \leq C | u_\varepsilon |_4^2 \leq C \varepsilon^{3/2} \to 0
\quad
\text{as }\varepsilon\to 0,
\]
proving the claim on $\phi_{u_\varepsilon}$ in Theorem \ref{th:main}.\\
Moreover, by standard regularity theory we can prove that $u_{\varepsilon}\in C^{2}(M)$, and so, by the compactness of $M$, it admits at least one maximum point.\\
To complete the proof of the second part of Theorem \ref{th:main}, it remains to show that, for $\varepsilon$ sufficiently small, such a maximum point is unique and to provide a description of the profile of $u_{\varepsilon}$.

Let us start with some preliminary results.

\begin{lemma} \label{lem:maxval}
Let $P\in M$ be a maximum point of $u_\varepsilon$. Then  $u_{\varepsilon}(P)\geq 1$.\footnote{Without the {\em normalization} of the constants assumed at the beginning of Section \ref{sec:prelim}, we get $u_{\varepsilon}(P)\geq \omega^{\frac1{p-2}}$.}
\end{lemma}
\begin{proof}
Since $P$ is a maximum point, $\Delta_{g}u_{\varepsilon}(P)\le0$. In addition, by Lemma \ref{lem:BP}, item (\ref{bLem:BP}), 
$$
u_\varepsilon^{p-1}(P)
=-\varepsilon^2\Delta_{g}u_{\varepsilon}(P)
+ u_{\varepsilon}(P) 
+ \phi_\varepsilon(P) u_{\varepsilon}(P)
\geq
u_{\varepsilon}(P),
$$
that concludes the proof.
\end{proof}

\begin{lemma}\label{lem:2max}
Suppose that there exist two maximum points $P_{\varepsilon}^{1},P_{\varepsilon}^{2}$ for $u_{\varepsilon}$ on $M$. Then $d_{g}(P_{\varepsilon}^{1},P_{\varepsilon}^{2})\rightarrow0$  as $\varepsilon\rightarrow0$.
\end{lemma}

\begin{proof}
By contradiction, let $\left\{ \varepsilon_{j}\right\}$ be a vanishing sequence such that $P_{\varepsilon_{j}}^{1}\rightarrow P^{1}\in M$,  $P_{\varepsilon_{j}}^{2}\rightarrow P^{2}\in M$, and $P^{1}\neq P^{2}$.\\
Let now
\[
Q_{\varepsilon_{j}}^{i}:=\exp_{P^{i}}^{-1}(P_{\varepsilon_{j}}^{i})\in B(0,r), \ i=1,2,
\]
and
\[
v_j^{1}(z):=u_{\varepsilon_{j}}\big(\exp_{P^{1}}(Q_{\varepsilon_{j}}^{1}+\varepsilon_{j}z)\big)
\text{ for }
|Q_{\varepsilon_{j}}^{1}+\varepsilon_{j}z|<r,
\]
$r$ being the injectivity radius of $M$.\\
Observe that, since, by definition, $Q_{\varepsilon_{j}}^{1}\rightarrow0$ as $\varepsilon_{j}\rightarrow0$,
\begin{equation}\label{palle}
B\Big(0,\frac{r}{2\varepsilon_{j}}\Big)
\subset B\left(-\frac{Q_{\varepsilon_{j}}^{1}}{\varepsilon_{j}},\frac{r}{\varepsilon_{j}}\right)
\text{ for } j \text{ large enough.}
\end{equation}
By Lemma \ref{lem:maxval} we have that
\begin{equation}\label{tildevj10}
v_j^{1}(0)=u_{\varepsilon_{j}}(P_{\varepsilon_{j}}^{1})\geq 1
\end{equation}
and, by \eqref{asymptg},
\begin{align*}
\int_{B(-Q_{\varepsilon_{j}}^{1}/\varepsilon_{j},r/\varepsilon_{j})} |\nabla v_{j}^{1}|^2 dz
&=
\frac{1}{\varepsilon_{j}}\int_{B(0,r)} | \nabla u_{\varepsilon_{j}}\big(\exp_{P^{1}}(y)\big)|^2 dy\\
&\leq
\frac{C}{\varepsilon_{j}}\int_{B(0,r)} g^{hl}(y) \partial_h u_{\varepsilon_{j}}\big(\exp_{P^{1}}(y)\big) \partial_l  u_{\varepsilon_{j}}\big(\exp_{P^{1}}(y)\big) |g_{P^1}(y)|^{1/2} dy\\
&=
\frac{C}{\varepsilon_{j}}\int_{B(P^1,r)} |\nabla_{g} u_{\varepsilon_{j}}|^2 d\mu_g
\leq
\frac{C}{\varepsilon_{j}}\int_{M} | \nabla_{g} u_{\varepsilon_{j}}|^2 d\mu_g
\end{align*}
and, analogously,
\begin{align*}
\int_{B(-Q_{\varepsilon_{j}}^{1}/\varepsilon_{j},r/\varepsilon_{j})} | v_{j}^{1}|^2 dz
&\leq
\frac{C}{\varepsilon_{j}^3}\int_{M} | u_{\varepsilon_{j}}|^2 d\mu_g.
\end{align*}
Thus, since $u_{\varepsilon_j}\in \mathcal{N}_{\varepsilon_j}$, by \eqref{rem:limsup}, we have
\begin{equation}\label{stimavj}
\begin{split}
C&\left(\int_{B(-Q_{\varepsilon_{j}}^{1}/\varepsilon_{j},r/\varepsilon_{j})} |\nabla v_{j}^{1}|^2 dy
+ \int_{B(-Q_{\varepsilon_{j}}^{1}/\varepsilon_{j},r/\varepsilon_{j})} | v_{j}^{1}|^2 dy\right)\\
&\le
\|u_{\varepsilon_{j}}\|_{\varepsilon_{j}}^{2}
=
\frac{2p}{p-2}J_{\varepsilon_{j}}(u_{\varepsilon_{j}}) - \frac{p-4}{2\varepsilon_{j}^3(p-2)} \int_M \phi_{u_{\varepsilon_j}} u_{\varepsilon_{j}}^2 d\mu_g
\leq
\frac{2p}{p-2} m_{\varepsilon_{j}}
\le
\frac{4p}{p-2} m_{\infty}.
\end{split}
\end{equation}
Now, let
\[
\bar{v}_{j}^{1}
:=
v_{j}^{1} \chi_{r}(|Q_{\varepsilon_{j}}^{1}+\varepsilon_{j}\cdot|).
\]
We have that $\bar{v}_j^1 \in H^{1}(\mathbb{R}^{3})$ and, since
\begin{align*}
|\nabla \bar{v}_{j}^{1}|_2^2
&=
\int_{\mathbb{R}^3} [\chi_{r}(|Q_{\varepsilon_{j}}^{1}+\varepsilon_{j}z|)]^2 |\nabla v_{j}^{1} (z) |^2 dz
+
\int_{\mathbb{R}^3} [v_{j}^{1} (z)]^2 |\nabla  \chi_{r}(|Q_{\varepsilon_{j}}^{1}+\varepsilon_{j}z|)|^2 dz\\
&\qquad
+
2\int_{\mathbb{R}^3} \chi_{r}(|Q_{\varepsilon_{j}}^{1}+\varepsilon_{j}z|) v_{j}^{1} (z)
\nabla v_{j}^{1} (z) \cdot \nabla \chi_{r}(|Q_{\varepsilon_{j}}^{1}+\varepsilon_{j}z|)dz\\
&\leq
\int_{B(-Q_{\varepsilon_{j}}^{1}/\varepsilon_{j},r/\varepsilon_{j})}  |\nabla v_{j}^{1} |^2 dz
+
\frac{4\varepsilon_{j}^2}{r^2} \int_{B(-Q_{\varepsilon_{j}}^{1}/\varepsilon_{j},r/\varepsilon_{j})} |v_{j}^{1} |^2  dz\\
&\qquad
+
\frac{4\varepsilon_{j}}{r}
\left(\int_{B(-Q_{\varepsilon_{j}}^{1}/\varepsilon_{j},r/\varepsilon_{j})}  |v_{j}^{1}|^2 dz\right)^{1/2}
\left(\int_{B(-Q_{\varepsilon_{j}}^{1}/\varepsilon_{j},r/\varepsilon_{j})}  |\nabla v_{j}^{1}|^2 dz\right)^{1/2}
\end{align*}
and
\[
|\bar{v}_{j}^{1}|_2^2
\leq
\int_{B(-Q_{\varepsilon_{j}}^{1}/\varepsilon_{j},r/\varepsilon_{j})} |v_{j}^{1}|^2 dz,
\]
by \eqref{stimavj}, the sequence $\{\bar{v}_{j}^{1}\}$ is bounded in $H^1(\mathbb{R}^3)$.\\
Thus, there exists $\bar{v}^{1}\in H^{1}(\mathbb{R}^{3})$ such that, up to a subsequence, $\bar{v}_{j}^{1}$ converges to $\bar{v}^{1} $  weakly in $H^{1}(\mathbb{R}^{3})$
and strongly in $L_{\text{loc}}^{t}(\mathbb{R}^{3})$ for $t\in[1,6)$.\\
Let now $\varphi\in C_{0}^{\infty}(\mathbb{R}^{3})$. Arguing as in \eqref{palle}, for $\varepsilon_{j}$ small we have that
$$
\operatorname{spt}(\varphi)
\subset B\left(0,\frac{r}{4\varepsilon_{j}}\right)
\subset B\left(-\frac{Q_{\varepsilon_{j}}^{1}}{\varepsilon_{j}},\frac{r}{2\varepsilon_{j}}\right)
$$
and so $\bar{v}_j^1=v_j^1$ on $\operatorname{spt}(\varphi)$.\\
Proceeding as in the proof of Proposition \ref{prop:baricentro}, if we define
\[
\varphi_j:M\to\mathbb{R},
\qquad
\varphi_{j}(x):=\varphi\left(\frac{\exp_{P^1}^{-1}(x)- Q_{\varepsilon_{j}}^1}{\varepsilon_{j}}\right)
\]
so that, for $j$ large $\operatorname{spt}(\varphi_j)\subset B_g(P^1,r)$, we have
\begin{align*}
\int_M \nabla_{g} u_{\varepsilon_j}\nabla_g \varphi_{j} d\mu_g
&=
\int_{B_g(P^1,r)} \nabla_{g} u_{\varepsilon_j}\nabla_g \varphi_{j} d\mu_g
=
\int_{B(0,r)} g_{P^1}^{il}(y) \partial_i \tilde{u}_{\varepsilon_j}(y)\partial_l \tilde{\varphi}_{j}(y) |g_{P^1}(y)|^{1/2}  d y
\\
&=
\varepsilon_j \int_{B(-Q_{\varepsilon_{j}}^1/\varepsilon_{j},r/\varepsilon_{j})} g_{P^1}^{il}(Q_{\varepsilon_{j}}^1+\varepsilon_j z) \partial_i v_j^1(z) \partial_l \varphi (z) |g_{P^1}(Q_{\varepsilon_{j}}^1+\varepsilon_j z)|^{1/2}  dz\\
&=
\varepsilon_j \int_{\operatorname{spt}(\varphi)} g_{P^1}^{il}(Q_{\varepsilon_{j}}^1+\varepsilon_j z) \partial_i \bar{v}_j^1(z) \partial_l \varphi (z) |g_{P^1}(Q_{\varepsilon_{j}}^1+\varepsilon_j z)|^{1/2}  dz
\end{align*}
and, arguing as in Proposition \ref{prop:baricentro} (proof of \eqref{convmisto}) and using again the boundedness of $\{u_{\varepsilon_{j}}\}$ in $H^1(M)$,
\[
\frac{1}{\varepsilon_j^3}\int_{M} \phi_{u_{\varepsilon_{j}}} u_{\varepsilon_{j}} \varphi_j d\mu_g
\to 0
\text{ as }j\to +\infty.
\]
Since $u_{\varepsilon_j}$'s are solutions of \eqref{eq:BPPnl}, we obtain
\begin{equation}\label{eq:Q1}
\begin{split}
0
&=
\frac{1}{\varepsilon_{j}^3}
\left[\varepsilon_j^2 \int_M \nabla_{g} u_{\varepsilon_j}\nabla_g \varphi_{j} d\mu_g
+\int_M  u_{\varepsilon_j} \varphi_{j} d\mu_g
+ \int_M \phi_{u_{\varepsilon_{j}}} u_{\varepsilon_j} \varphi_{j} d\mu_g
-  \int_M u_{\varepsilon_j}^{p-1} \varphi_{j} d\mu_g\right]\\
&=
\int_{\operatorname{spt}(\varphi)} g_{P^1}^{il}(Q_{\varepsilon_{j}}^1+\varepsilon_j z) \partial_i \bar{v}_j^1(z) \partial_l \varphi (z)
|g_{P^1}(Q_{\varepsilon_{j}}^1+\varepsilon_j z)|^{1/2}  dz
\\
&\quad
+
\int_{\operatorname{spt}(\varphi)}
\bar{v}_{j}^{1}(z)\varphi(z)|g_{P^{1}}(Q_{\varepsilon_{j}}^{1}+\varepsilon_{j}z)|^{1/2}dz
-\int_{\operatorname{spt}(\varphi)}(\bar{v}_{j}^{1}(z))^{p-1}\varphi(z)|g_{P^{1}}(Q_{\varepsilon_{j}}^{1}+\varepsilon_{j}z)|^{1/2}dz
\\
&\quad
+o_j(1).
\end{split}
\end{equation}
Thus, passing to the limit as $j\to+\infty$ in \eqref{eq:Q1}, by \eqref{asymptg}, we deduce that for all $\varphi\in C_{0}^{\infty}(\mathbb{R}^{3})$
\[
0
=\int_{\operatorname{spt}(\varphi)}\left[\nabla\bar{v}^{1}(z)\nabla\varphi(z)
+\bar{v}^{1}(z)\varphi(z)
-(\bar{v}^{1}(z))^{p-1}\varphi(z)\right]dz,
\]
so that $\bar{v}^{1}$ is a weak solution of
\[
-\Delta\bar{v}^{1}+\bar{v}^{1}=(\bar{v}^{1})^{p-1}\text{ on }\mathbb{R}^{3}.
\]
By a bootstrap argument (see Appendix \ref{bootapp}) we have that, up to a subsequence, $\bar{v}_{j}^{1}\rightarrow\bar{v}^{1}$ in $C_{\text{loc}}^{2}(\mathbb{R}^{3})$. \\
Then, since by \eqref{tildevj10}, $\bar{v}^{1}(0)\ge 1$, $\bar{v}^{1}$ is not the trivial solution zero.
Hence, $\bar{v}^{1}=U$.\\
Of course we can repeat the same argument for $P_{\varepsilon_{j}}^{2}$.\\
Now, let $\bar{R}$ such that
\[
\int_{B(0,\bar{R})}[|\nabla U|^{2}+  U^{2}]dy>\frac{p}{p-2}(m_{\infty}+\delta).
\]
For $\varepsilon_{j}$ sufficiently small, such that $2\varepsilon_{j}\bar{R}\leq d_{g}(P^{1},P^{2})$, we have, since $p>4$,
\begin{equation*}
		J_{\varepsilon_{j}}(u_{\varepsilon_{j}})
		 =
		\frac{p-2}{2p}\|u_{\varepsilon_{j}}\|_{\varepsilon_{j}}^{2}
		+ \frac{p-4}{4p\varepsilon_j^3} \int_M \phi_{u_{\varepsilon_j}} u_{\varepsilon_{j}}^2 d\mu_g
		 \geq
		\frac{p-2}{2p\varepsilon_j^3} \int_{B_{g}(P^{1},\varepsilon_{j}\bar{R})\cup B_{g}(P^{2},\varepsilon_{j}\bar{R})} [\varepsilon_j^{2}|\nabla_{g}u_{\varepsilon_{j}}|^{2}+ u_{\varepsilon_{j}}^{2}] d\mu_g.
\end{equation*}
Then, since, for $h=1,2$,
\begin{align*}
\frac{1}{\varepsilon_j^3} \int_{B_{g}(P^{h},\varepsilon_{j}\bar{R})} [\varepsilon_j^{2}|\nabla_{g}u_{\varepsilon_{j}}|^{2}
+ u_{\varepsilon_{j}}^{2}] d\mu_g
&=
\int_{B(0,\bar{R})} [|\nabla \tilde{u}_{\varepsilon_{j}}|^{2}+ \tilde{u}_{\varepsilon_{j}}^{2}] dz + o_j(1)\\
&=
\int_{B(0,\bar{R})} [|\nabla \bar{v}_{j}^h|^{2}+ (\bar{v}_{j}^h)^{2}] dz + o_j(1)
\to
\int_{B(0,\bar{R})} [|\nabla U|^{2}+ U^{2}] dz
\end{align*}
we get $J_{\varepsilon_{j}}(u_{\varepsilon_{j}})
>m_{\infty}+\delta$ which leads us to a contradiction.
\end{proof}
\begin{lemma}
	\label{lem:unimax}
	If $\varepsilon$ is sufficiently small, $u_{\varepsilon}$
	has a unique maximum point.
\end{lemma}
\begin{proof}
	Suppose that there exists a sequence $\varepsilon_{j}\rightarrow0$ such that $u_{\varepsilon_{j}}$ has at least two maximum points $P_{\varepsilon_{j}}^{1}$ 	and $P_{\varepsilon_{j}}^{2}$.
	By Lemma \ref{lem:2max} we know that
	$d_{g}(P_{\varepsilon_{j}}^{1},P_{\varepsilon_{j}}^{2})\rightarrow0$.\\
	We have also that 
	\begin{equation}
		\lim_{j}\frac{1}{\varepsilon_{j}}d_{g}(P_{\varepsilon_{j}}^{1},P_{\varepsilon_{j}}^{2})=+\infty. \label{eq:limmax}
	\end{equation}
	Indeed, suppose by contradiction that 	 $d_{g}(P_{\varepsilon_{j}}^{1},P_{\varepsilon_{j}}^{2})\le C\varepsilon_{j}$
	for some $C>0$ and let
	\[
	w_{\varepsilon_{j}}:=u_{\varepsilon_{j}}(\exp_{P_{\varepsilon_{j}}^{1}}(\varepsilon_{j}\cdot))\text{ in }B(0,2C).
	\]
	Thus, for $j$ large enough, $w_{\varepsilon_{j}}$ has two maximum points in $B(0,2C)$.\\
	Moreover, arguing as in Lemma \ref{lem:2max}, $w_{\varepsilon_{j}}\rightarrow U$ in $C_{\text{loc}}^{2}(\mathbb{R}^{3})$ and the two maximum points of $w_{\varepsilon_{j}}$ collapse in $0$. Thus $0$ should be a degenerate critical point for $U$. This is a contradiction and \eqref{eq:limmax} is proved.\\
	Now, in light of \eqref{eq:limmax}, we have that, fixed $\rho>0$, then $B_{g}(P_{\varepsilon_{j}}^{1},\rho\varepsilon_{j})\cap B_{g}(P_{\varepsilon_{j}}^{2},\rho\varepsilon_{j})=\emptyset$ for $j$ large. Then we proceed as in the final part of the proof of Lemma \ref{lem:2max} obtaining $J_{\varepsilon_{j}}(u_{\varepsilon_{j}})>m_\infty+\delta$ which leads us to a contradiction.\end{proof}

We conclude this section with the profile description of $u_\varepsilon$.

\begin{lemma} \label{lem:stime}
As $\varepsilon\rightarrow0$, for any $\rho>0$, $\|u_{\varepsilon}-W_{P_{\varepsilon},\varepsilon}\|_{C^{2}(B_{g}(P_{\varepsilon},\varepsilon\rho))}\rightarrow0$, and $\|u_{\varepsilon}-W_{P_{\varepsilon},\varepsilon}\|_{L^{\infty}(M)}\rightarrow0$.
\end{lemma}
\begin{proof}
	By the $C^{2}$ convergence proved in Lemma \ref{lem:2max} we have that, given $\rho>0$,
	\[
	\|u_{\varepsilon}-W_{P_{\varepsilon},\varepsilon}\|_{C^{2}(B_{g}(P_{\varepsilon},\varepsilon\rho))}=\|u_{\varepsilon}(\exp_{P_{\varepsilon}}(\varepsilon z))-U(z)\|_{C^{2}(B(0,\rho))}\rightarrow0
	\quad \text{ as } \varepsilon\rightarrow0.
	\]
	Moreover, since, by Lemma \ref{lem:unimax}, $u_{\varepsilon}$ has a unique maximum point $P_\varepsilon$, we have that, for any $\rho>0$,
	\[
	\max_{M\setminus B_{g}(P_{\varepsilon},\varepsilon\rho)}u_{\varepsilon}
	=\max_{\partial B_{g}(P_{\varepsilon},\varepsilon\rho)}u_{\varepsilon}
	=\max_{|z|=\rho}U(z)+o(1)
	\le ce^{-\alpha\rho}+o(1)
	\]
	for some constant $c,\alpha>0$ as $\varepsilon\rightarrow0$. This proves the claim.\end{proof}

\section{A further solution}\label{sec:further}
In this section we prove that, since $\operatorname{cat}(M)>1$, for $\varepsilon$ small enough, there exists a further nonconstant solution with higher energy, concluding the proof of Theorem \ref{th:main}.

To this end, it is enough to show that, for $\varepsilon$ small, it is possible to construct a contractible set $T_\varepsilon \subset \mathcal{N}_\varepsilon$ with $
\sup_{T_\varepsilon} J_\varepsilon \leq C$, where the constant $C$ does not depend on $\varepsilon$.\\
Then, since $ \mathcal{N}_\varepsilon\cap J_\varepsilon^{m_\infty+\delta}$ is not contractible due to
\[
\operatorname{cat}(\mathcal{N}_\varepsilon\cap J_\varepsilon^{m_\infty+\delta})\geq \operatorname{cat} (M)>1,
\]
we get that there exists a solution $w_\varepsilon$ with $m_\infty+\delta<J_\varepsilon(w_\varepsilon)$ applying the second part of Theorem \ref{prelim}.
Given a positive function $V\in H^1(\mathbb{R}^3)$ and a point $\xi_0\in M$, we define a function $v_\varepsilon\in H^1(M)$ as 
\[
v_{\varepsilon}:=V_{\varepsilon}(\exp_{\xi_0}^{-1}\cdot)\chi(|\exp_{\xi_0}^{-1}\cdot|)
\]
where $\chi$ is as in definition of the exponential map and $V_\varepsilon=V(\cdot/\varepsilon)$, and the cone $C_\varepsilon\subset H^1(M)$ as
$$
C_\varepsilon:=\{u=\theta v_\varepsilon +(1-\theta)W_{\xi,\varepsilon},\  \theta\in [0,1],\ \xi\in M\}.
$$
Easily we have that $C_\varepsilon$ is a compact contractible set in $H^1(M)$.\\
Since, by (\ref{tucont}) in Lemma \ref{lemNehari}, the map $u\mapsto t_u$ is continuous, we can project $C_\varepsilon$ on the Nehari manifold 
$\mathcal{N}_\varepsilon$ obtaining the following compact and contractible set 
$$
T_\varepsilon:=\{t_u u :\ u\in C_\varepsilon\}.
$$
In addiction, 
since
\begin{align*}
	\frac{1}{\varepsilon} \int_M | \nabla_g v_{\varepsilon} |^2 d \mu_g
	& = \frac{1}{\varepsilon} \int_{B (0, r)} g^{{ij}}_{\xi_0} (y)
	\partial_i \tilde{v}_{\varepsilon} (y) \partial_j \tilde{v}_{\varepsilon}
	(y) | g_{\xi_0} (y) |^{1 / 2} {dy}\\
	& = \int_{B (0, r / \varepsilon)} g^{{ij}}_{\xi_0} (\varepsilon z)
	\chi^2 (| \varepsilon z |) \partial_i V (z) \partial_j V (z) | g_{\xi_0}
	(\varepsilon z) |^{1 / 2} {dz}\\
	& \qquad + \varepsilon  \int_{B (0, r / \varepsilon)} g^{{ij}}_{\xi_0}
	(\varepsilon z) \chi (| \varepsilon z |) \chi' (| \varepsilon z |) V (z)
	\partial_i V (z) \frac{z_j}{| z |} | g_{\xi_0} (\varepsilon z) |^{1 / 2}
	{dz}\\
	& \qquad + \varepsilon^2 \int_{B (0, r / \varepsilon)}
	g^{{ij}}_{\xi_0} (\varepsilon z) [\chi' (| \varepsilon z |)]^2
	\frac{z_i z_j}{| z |^2} V^2 (z) | g_{\xi_0} (\varepsilon z) |^{1 / 2}
	{dz}\\
	& \leq  C \Big[| \nabla V|_2^2 + \frac{2}{r}  \varepsilon |
	\nabla V|_2 |V|_2 + \frac{4}{r^2}  \varepsilon^2 |V|_2^2\Big],
\end{align*}
arguing as in (\ref{1lem:limiti}) of Lemma \ref{lem:limiti}, we have 
\[
\|\theta v_{\varepsilon}+(1-\theta)W_{\xi,\varepsilon}\|_{\varepsilon}\le\|v_{\varepsilon}\|_{\varepsilon}+\|W_{\xi,\varepsilon}\|_{\varepsilon}
\le
C\left(\|V\|_{H^{1}}+\|U\|_{H^{1}}\right),
\]
and
\[
|\theta v_{\varepsilon}+(1-\theta)W_{\xi,\varepsilon}|_{\varepsilon,p}
\le C\left(|V|_p+|U|_p\right).
\]
Furthermore, since $v_{\varepsilon}\ge0$ and $W_{q,\varepsilon}\ge0$, it holds
\[
\theta v_{\varepsilon}(x)+(1-\theta)W_{\xi,\varepsilon}(x)
\ge
\max\left\{ \theta v_{\varepsilon}(x),(1-\theta) W_{\xi,\varepsilon}(x) \right\},
\text{ for all } x\in M,
\]
so that
\[ | \theta v_{\varepsilon} + (1 - \theta) W_{\xi,\varepsilon}|_p^p \geq \max \{ | \theta v_{\varepsilon}|_p^p, | (1 -
\theta) W_{\xi,\varepsilon}|_p^p \} . \]
Then, if $\theta \geq \frac{1}{2}$, 
\[ \max \{ | \theta v_{\varepsilon}|_p^p, | (1 - \theta) W_{\xi,\varepsilon}|_p^p \} \geq \theta^p
|v_{\varepsilon}|_p^p \geq \frac{1}{2^p} |v_{\varepsilon}|_p^p \geq \frac{1}{2^p} \min \{
|v_{\varepsilon}|_p^p, |W|_p^p \} \]
and, if $\theta \leq \frac{1}{2}$, 
\[ \max \{ | \theta v_{\varepsilon}|_p^p, | (1 - \theta) W_{\xi,\varepsilon}|_p^p \} \geq (1 - \theta)^p
|W_{\xi,\varepsilon}|_p^p \geq \frac{1}{2^p} |W_{\xi,\varepsilon}|_p^p \geq \frac{1}{2^p} \min \{
|v_{\varepsilon}|_p^p, |W_{\xi,\varepsilon}|_p^p \} . \]
Thus, arguing as in (\ref{2lem:limiti}) of Lemma \ref{lem:limiti} and using it,
\[
| \theta v_{\varepsilon} + (1 - \theta) W_{\xi,\varepsilon}|_{\varepsilon,p}^p
\geq
\frac{1}{2^p} \min \{ |v_{\varepsilon}|_{\varepsilon,p}^p, |W_{\xi,\varepsilon}|_{\varepsilon,p}^p\}
\geq
\frac{1}{4^p} \min \{ |V|_{p}^p, |U|_{p}^p\}.
\]
Analogously, we get also
\[
\|\theta v_{\varepsilon}+(1-\theta)W_{q,\varepsilon}\|_{\varepsilon}^2
\geq
\frac{1}{4^2}\min\left\{ |V|_{2}^2,|U|_{2}^2\right\} .
\]
Finally, by (\ref{aLem:BP}) in Lemma \ref{lem:BP} and \eqref{eq:punto 3} we have 
$$
\frac1{\varepsilon^3}G(\theta v_{\varepsilon}+(1-\theta)W_{q,\varepsilon})
\le \varepsilon^3 C |\theta v_{\varepsilon}+(1-\theta)W_{q,\varepsilon}|^4_{2,\varepsilon}
$$
Hence, using \eqref{eq:Nt}, we have that there exist $c_1,c_2>0$, independent of $\varepsilon$, such that
$$c_1<t_{\theta v_{\varepsilon}+(1-\theta)W_{q,\varepsilon}}<c_2$$
and so we can conclude now, since, for any $t_uu\in T_\varepsilon$, we have 
\[
J_{\varepsilon}(t_uu)=\frac{1}{4}t_u^{2}\|u\|_{\varepsilon}^{2}+\left(\frac{1}{4}-\frac{1}{p}\right)t_u^{p}|u^{+}|_{p,\varepsilon}^{p}\le C.
\]
Finally, we check that the solution found is not a constant function.
\\
Notice that, if $u\equiv c_* \neq 0$,  then $\phi\equiv 4\pi c_*^2$, and $c_*^{p-2}-4\pi q^2 c_*^2-\omega=0$, so that $c_*\geq C>0$.
Since $c_*\in \mathcal{N}_\varepsilon$, then
$$
J_\varepsilon(c_*)=\frac{1}{4}\|c_*\|_{\varepsilon}^{2}+\left(\frac{1}{4}-\frac{1}{p}\right)|c_*|_{p,\varepsilon}^{p}=
\frac{\mu_g (M)}{\varepsilon^3}\left[ \frac{c_*^2}{4}+\left(\frac{1}{4}-\frac{1}{p}\right)c_*^p\right]
\rightarrow+\infty
\text{ as } \varepsilon\to0
$$
and this is not possible since $\sup_{T_\varepsilon} J_\varepsilon \leq C$.

\appendix
\section{Ekeland Principle}\label{EVP}

	In this Appendix we want to show that \eqref{eq:ps} holds.\\
	To this end, we proceed as in \cite[Lemma 5.4]{BBM}, applying the Ekeland Principle (see \cite[Chapter 4]{deFig}) as follows:\\
	for every $\theta,\iota>0$ and $u\in J_\varepsilon^{m_\varepsilon+\theta/2}$, there exists $u_\iota\in \mathcal{N}_\varepsilon$ such that
	\[
	J_\varepsilon (u_\iota)<J_\varepsilon (u),
	\quad
	\|u_\iota-u\|_\varepsilon<\iota,
	\quad
	J_\varepsilon (u_\iota)<J_\varepsilon (v)+\frac{\theta}{\iota}\|u_\iota-u\|_\varepsilon<\iota
	\text{ for all }v\in\mathcal{N}_\varepsilon.
	\]
	Thus, for every $k$, taking $\theta=4\delta_k$ and $\iota=4\sqrt{\delta_k}$, there exists $\tilde{u}_k\in \mathcal{N}_{\varepsilon_{k}}$ such that
	\begin{equation}
		\label{EkPrinc}
		J_{\varepsilon_{k}}(\tilde{u}_k) < J_{\varepsilon_{k}}(u_k),
		\quad
		\| \tilde{u}_k - u_k\|_{\varepsilon_{k}}<4\sqrt{\delta_k},
		\quad
		J_{\varepsilon_{k}}(\tilde{u}_k)< J_{\varepsilon_{k}}(v)+\sqrt{\delta_k}\| \tilde{u}_k - v\|_{\varepsilon_{k}} \text{ for all } v\in \mathcal{N}_{\varepsilon_{k}}.
	\end{equation}
	The boundedness of $\{\|u_k\|_{\varepsilon_{k}}\}$ and \eqref{EkPrinc} implies that $\{\|\tilde{u}_k\|_{\varepsilon_{k}}\}$ is bounded too.\\
	Observe, moreover, that, for every $\xi\in T_{\tilde{u}_k}\mathcal{N}_{\varepsilon_{k}}$, there exists a smooth curve $\gamma:[a,b]\to \mathcal{N}_{\varepsilon_{k}}$, with $a<0<b$, such that
	\[
	\gamma(0)=\tilde{u}_k
	\text{ and }
	\gamma'(0)=\xi
	\]
	(see e.g. \cite{AM}).
	\\
	Let $\{t_n\}\subset\mathbb{R}$ such that $t_n \to 0$.\\
	Since
	\[
	J_{\varepsilon_{k}}(\gamma(t_n))
	=
	J_{\varepsilon_{k}}(\gamma(0))
	+ J_{\varepsilon_{k}}'(\gamma(0))[\gamma'(0)]t_n
	+O(t_n^2)
	=
	J_{\varepsilon_{k}}(\tilde{u}_k)
	+ J_{\varepsilon_{k}}'(\tilde{u}_k)[\xi] t_n
	+O(t_n^2)
	\]
	and
	\[
	\| \tilde{u}_k - \gamma(t_n)\|_{\varepsilon_{k}}
	= \| \gamma(0) - \gamma(t_n)\|_{\varepsilon_{k}}
	= \| \gamma'(0)t_n +O(t_n^2)\|_{\varepsilon_{k}}
	= \| \xi t_n +O(t_n^2)\|_{\varepsilon_{k}},
	\]
	by the Ekeland Variational Principle, we get
	\[
	\sqrt{\delta_k}>
	\frac{J_{\varepsilon_{k}}(\tilde{u}_k)-J_{\varepsilon_{k}}(\gamma(t_n))}{\| \tilde{u}_k - \gamma(t_n)\|_{\varepsilon_{k}}}
	=
	\frac{t_n}{|t_n|}
	\frac{J_{\varepsilon_{k}}'(\tilde{u}_k)[\xi]+O(t_n)}{\| \xi+O(t_n)\|_{\varepsilon_{k}}}.
	\]
	Considering the left and right limits as $t_n\to 0$ we can conclude that
	\begin{equation}
		\label{J'tan}
		|J_{\varepsilon_{k}}'(\tilde{u}_k)[\xi]|\leq \sqrt{\delta_k} \| \xi\|_{\varepsilon_{k}}
		\text{ for all } \xi\in T_{\tilde{u}_k}\mathcal{N}_{\varepsilon_{k}}.
	\end{equation}
	Let now $\varphi\in H^1(M)$ be arbitrary.\\
	Since $\tilde{u}_k\in \mathcal{N}_{\varepsilon_{k}}$, by (\ref{disc0}) in Lemma \ref{lemNehari},
	\begin{equation}\label{enneprimo}
		N'_{\varepsilon_{k}}(\tilde{u}_k)[\tilde{u}_k]
		= -2 \| \tilde{u}_k\|_{\varepsilon_{k}}^2
		-(p-4)|\tilde{u}_k^+|_{p,\varepsilon_{k}}^p
		\leq -C
		<0.
	\end{equation}
	Then $\tilde{u}_k\notin T_{\tilde{u}_k} \mathcal{N}_{\varepsilon_{k}}$.
	Thus there exists $\lambda,\mu\in\mathbb{R}$ and $\xi \in T_{\tilde{u}_k} \mathcal{N}_{\varepsilon_{k}}$ such that
	$\varphi =\lambda\xi + \mu \tilde{u}_k$.\\
	Observe that, by \eqref{enneprimo}, there exists $C\in (0,1)$ such that for all $u\in\mathcal{N}_{\varepsilon_{k}}, \xi \in T_u \mathcal{N}_{\varepsilon_{k}}$, $ |\langle\xi,u\rangle_{\varepsilon_{k}}|\leq C\|\xi\|_{\varepsilon_{k}} \|u\|_{\varepsilon_{k}}$. Then a straightforward calculation shows that there exists $C>0$ such that $\|\lambda\xi\|_{\varepsilon_k} \leq C\| \varphi\|_{\varepsilon_k}$.\\
	Then, by \eqref{J'tan}, we get that for every $\varphi \in H^1(M)$
	\begin{equation}
		\label{eq:pstilde}
		|J_{\varepsilon_{k}}'(\tilde{u}_k)[\varphi]|
		= |\lambda J_{\varepsilon_{k}}'(\tilde{u}_k)[\xi]|
		\leq \sqrt{\delta_k} \| \lambda \xi\|_{\varepsilon_{k}}
		\leq C \sqrt{\delta_k} \| \varphi\|_{\varepsilon_{k}}.
	\end{equation}
	Hence, we can conclude. Indeed, since
	\[
	|J_{\varepsilon_{k}}'(u_k)[\varphi]|
	\leq
	|J_{\varepsilon_{k}}'(u_k)[\varphi]-J_{\varepsilon_{k}}'(\tilde{u}_k)[\varphi]|
	+|J_{\varepsilon_{k}}'(\tilde{u}_k)[\varphi]|,
	\]
	by \eqref{eq:pstilde}, it is enough to prove that
	\begin{equation}
		\label{claimPS}
		|J_{\varepsilon_{k}}'(u_k)[\varphi]-J_{\varepsilon_{k}}'(\tilde{u}_k)[\varphi]|
		\leq C\sqrt{\delta_{k}}\|\varphi\|_{\varepsilon_{k}}.
	\end{equation}
	This follows observing that
	\begin{align*}
		|J_{\varepsilon_{k}}'(u_k)[\varphi]-J_{\varepsilon_{k}}'(\tilde{u}_k)[\varphi]|
		&=
		\Big|\frac{1}{\varepsilon_{k}^3}\Big[\varepsilon_{k}^2\int_M \nabla_{g} (u_k-\tilde{u}_k)\nabla_{g}\varphi d\mu_{g}
		+\omega \int_M (u_k-\tilde{u}_k)\varphi d\mu_{g}\\
		&\qquad
		+q^2 \int_M (\phi_{u_k}u_k-\phi_{\tilde{u}_k}\tilde{u}_k)\varphi d\mu_{g}
		-\int_M (|u_k^+|^{p-2}u_k^+ - |\tilde{u}_k^+|^{p-2}\tilde{u}_k^+)\varphi d\mu_g\Big]
		\Big|.
	\end{align*}
	Then, by H\"older inequality and \eqref{EkPrinc},
	\[
	\Big|\frac{1}{\varepsilon_{k}^3}\Big[\varepsilon_{k}^2\int_M \nabla_{g} (u_k-\tilde{u}_k)\nabla_{g}\varphi d\mu_{g}
	+\omega \int_M (u_k-\tilde{u}_k)\varphi d\mu_{g}	\Big|
	\leq
	\|u_k-\tilde{u}_k\|_{\varepsilon_{k}} \|\varphi\|_{\varepsilon_{k}}
	< 4 \sqrt{\delta_{k}} \|\varphi\|_{\varepsilon_{k}}.
	\]
	Moreover
	\begin{equation}\label{phitwoterms}
		\frac{1}{\varepsilon_{k}^3}\Big|\int_M (\phi_{u_k}u_k-\phi_{\tilde{u}_k}\tilde{u}_k)\varphi d\mu_{g}\Big|
		\leq
		\frac{1}{\varepsilon_{k}^3}\Big|\int_M \phi_{u_k} (u_k-\tilde{u}_k)\varphi d\mu_{g}\Big|
		+\frac{1}{\varepsilon_{k}^3}\Big|\int_M (\phi_{u_k}-\phi_{\tilde{u}_k})\tilde{u}_k \varphi d\mu_{g}\Big|.
	\end{equation}
	Considering the first term in the right hand side of \eqref{phitwoterms}, by Lemma \ref{lem:BP}, Sobolev embedding $H^2(M)\subset C^0(M)$, H\"older inequality, \eqref{imb}, the boundedness of $\{\|u_k\|_{\varepsilon_{k}}\}$, and \eqref{EkPrinc},
	\[
	\frac{1}{\varepsilon_{k}^3}\Big|\int_M \phi_{u_k}(u_k-\tilde{u}_k)\varphi d\mu_{g}\Big|
	\leq
	|\phi_{u_k}|_{C^0} |u_k-\tilde{u}_k|_{2,\varepsilon_{k}} |\varphi|_{2,\varepsilon_{k}}
	\leq
	C | u_k |_2^2 \|u_k-\tilde{u}_k\|_{\varepsilon_{k}} \|\varphi\|_{\varepsilon_{k}}
	\leq C \varepsilon_{k}^3 \sqrt{\delta_{k}} \|\varphi\|_{\varepsilon_{k}}.
	\]
%
	To estimate the second term in the right hand side of \eqref{phitwoterms}, first observe that
	\[
	-\Delta_g (\phi_{u_k}-\phi_{\tilde{u}_k})
	+ a^2 \Delta_g^2 (\phi_{u_k}-\phi_{\tilde{u}_k})
	+ (\phi_{u_k}-\phi_{\tilde{u}_k})
	= 4\pi (u_k^2-\tilde{u}_k^2).
	\]
	Then, using also Sobolev and H\"older inequalities, \eqref{imb}, \eqref{EkPrinc}, and the boundedness of $\{\|u_k\|_{\varepsilon_{k}}\}$ and $\{\|\tilde{u}_k\|_{\varepsilon_{k}}\}$,
	\begin{align*}
	\| \phi_{u_k}-\phi_{\tilde{u}_k}\|_{H^2}^2
	&= 4\pi \int_M (\phi_{u_k}-\phi_{\tilde{u}_k}) (u_k^2-\tilde{u}_k^2) d\mu_{g}
	\leq C \|\phi_{u_k}-\phi_{\tilde{u}_k}\|_{H^2} \int_M |u_k^2-\tilde{u}_k^2|d\mu_g\\
	&
	\leq C \| \phi_{u_k}-\phi_{\tilde{u}_k}\|_{H^2} |u_k-\tilde{u}_k|_2 |u_k+\tilde{u}_k|_2
	\leq C \varepsilon_{k}^3 \| \phi_{u_k}-\phi_{\tilde{u}_k}\|_{H^2}  \|u_k-\tilde{u}_k\|_{\varepsilon_{k}}.
	\end{align*}
	Thus, using also Sobolev imbeddings, H\"older inequality, and \eqref{EkPrinc},
	\begin{align*}
		\frac{1}{\varepsilon_{k}^3}\Big|\int_M (\phi_{u_k}-\phi_{\tilde{u}_k}) \tilde{u}_k \varphi d\mu_{g}\Big|
		&\leq
		\frac{1}{\varepsilon_{k}^3}
		|\phi_{u_k}-\phi_{\tilde{u}_k}|_{C^0}
		|\tilde{u}_k|_2 |\varphi|_2
		\leq
		\frac{C}{\varepsilon_{k}^3} \|\phi_{u_k}-\phi_{\tilde{u}_k}\|_{H^2} |\tilde{u}_k|_2 |\varphi|_2\\
		&\leq
		C \varepsilon_{k}^3 \|u_k-\tilde{u}_k\|_{\varepsilon_{k}}
		|\tilde{u}_k|_{2,\varepsilon_{k}} |\varphi|_{2,\varepsilon_{k}}
		\leq
		C \varepsilon_{k}^3 \sqrt{\delta_{k}} \|\varphi\|_{\varepsilon_{k}}.
	\end{align*}
	Finally, by Lagrange Theorem,
	H\"older inequality, \eqref{imb}, boundedness of $\{\|u_k\|_{\varepsilon_{k}}\}$ and $\{\|\tilde{u}_k\|_{\varepsilon_{k}}\}$, \eqref{EkPrinc}, we have
	\begin{align*}
		\Big|\frac{1}{\varepsilon_{k}^3}\int_M (|u_k^+|^{p-2}u_k^+ - |\tilde{u}_k^+|^{p-2}\tilde{u}_k^+)\varphi d\mu_g
		\Big|
		&\leq
		\frac{p-1}{\varepsilon_{k}^3}
		\int_M |\theta_k u_k^+ +(1-\theta_k ) \tilde{u}_k^+|^{p-2}
		|u_k^+ - \tilde{u}_k^+| |\varphi| d\mu_g\\
		&\leq
		C
		(|u_k^+|_{p,\varepsilon_{k}}^{p-2} + |\tilde{u}_k^+|_{p,\varepsilon_{k}}^{p-2})
		|u_k^+ - \tilde{u}_k^+|_{p,\varepsilon_{k}}
		|\varphi|_{p,\varepsilon_{k}}
		\leq
		C \sqrt{\delta_{k}} \|\varphi\|_{\varepsilon_{k}},
	\end{align*}
	completing the proof of \eqref{claimPS}.

\section{Bootstrap argument}\label{bootapp}

In this section, through a classical bootstrap argument, we prove that $\{\bar{v}_{j}^{1}\} \subset C^{2}(B(0,R/2))$ and that it is bounded in $C^{2}(B(0,R/2))$.\\
Let $R>0$.\\
Observe that, arguing as in \eqref{palle}, for $j$ large,
\begin{equation}\label{pallebis}
	B(0,R)
	\subset B\Big(0,\frac{r}{4\varepsilon_{j}}\Big)
	\subset B\Big(-\frac{Q_{\varepsilon_{j}}^1}{\varepsilon_{j}},\frac{r}{2\varepsilon_{j}}\Big).
\end{equation}
Thus, in $B(0,R)$, since $u_{\varepsilon_{j}}$ is a solution of \eqref{eq:BPPnl}, we have
\begin{equation}\label{glaplueps}
	\begin{split}
		-\varepsilon_{j}^2 (\Delta_g u_{\varepsilon_{j}}) \big(\exp_{P^{1}}(Q_{\varepsilon_{j}}^{1}+\varepsilon_{j}z)\big)
		&=
		- u_{\varepsilon_{j}} \big(\exp_{P^{1}}(Q_{\varepsilon_{j}}^{1}+\varepsilon_{j}z)\big)
		+ (u_{\varepsilon_{j}})^{p-1}\big(\exp_{P^{1}}(Q_{\varepsilon_{j}}^{1}+\varepsilon_{j}z)\big)
		\\
		&\qquad
		- \phi_{u_{\varepsilon_j}} \big(\exp_{P^{1}}(Q_{\varepsilon_{j}}^{1}+\varepsilon_{j}z)\big) u_{\varepsilon_{j}} \big(\exp_{P^{1}}(Q_{\varepsilon_{j}}^{1}+\varepsilon_{j}z)\big).
	\end{split}
\end{equation}
But, since
\begin{multline*}
	(\Delta_g u_{\varepsilon_{j}}) \big(\exp_{P^{1}}(Q_{\varepsilon_{j}}^{1}+\varepsilon_{j}z)\big)
	=
	\frac{1}{\varepsilon_{j}^2} g_{P^{1}}^{il}(Q_{\varepsilon_{j}}^{1}+\varepsilon_{j}z)
	\partial_{il} \bar{v}_j^1 (z)
	\\
	+
	\frac{1}{\varepsilon_{j}^2 |g_{P^1}(Q_{\varepsilon_{j}}^{1}+\varepsilon_{j}z)|^{1/2}}
	\partial_{l}\left(g_{P^{1}}^{il}(Q_{\varepsilon_{j}}^{1}+\varepsilon_{j}\cdot)|g_{P^{1}}(Q_{\varepsilon_{j}}^{1}+\varepsilon_{j}\cdot)|^{1/2}\right) (z)
	\partial_{i} \bar{v}_j^1(z),
\end{multline*}
\eqref{glaplueps} reads as
\begin{equation}\label{bootstrap}
	-g_{P^{1}}^{il}(Q_{\varepsilon_{j}}^{1}+\varepsilon_{j}z) \partial_{il} \bar{v}_j^1 (z)
	=
	f_j(z),
\end{equation}
where
\begin{align*}
	f_j(z)
	&:=
	\frac{1}{|g_{P^1}(Q_{\varepsilon_{j}}^{1}+\varepsilon_{j}z)|^{1/2}}
	\partial_{l}\left(g_{P^{1}}^{il}(Q_{\varepsilon_{j}}^{1}+\varepsilon_{j}\cdot)|g_{P^{1}}(Q_{\varepsilon_{j}}^{1}+\varepsilon_{j}\cdot)|^{1/2}\right) (z)
	\partial_{i} \bar{v}_j^1(z)
	- \bar{v}_j^1(z)
	\\
	&\qquad
	- \bar{\phi}_j(z) \bar{v}_j^1(z)
	+ (\bar{v}_j^1)^{p-1}(z)
\end{align*}
and $\bar{\phi}_j(z):=\phi_{u_{\varepsilon_j}} \big(\exp_{P^{1}}(Q_{\varepsilon_{j}}^{1}+\varepsilon_{j}z)\big)$.\\
Let us show that
$\bar{v}_j^1 \in C^{0,\alpha}(B(0,R))$ for some $\alpha\in(0,1)$.\\
Let us consider equation \eqref{bootstrap} in $B(0,R)$ and let $q_0:=6$. Since $p \in (4, 6)$, then $q_0/(p-1)\in(6/5,2)$.
Thus
\[
\min \left\{ 2, \frac{q_0}{p-1} \right\} =  \frac{q_0}{p-1}
\]
and so, using the boundedness of $\{\bar{v}_j^1\}$ in $H^1(\mathbb{R}^3)$ and of $\{\|u_{\varepsilon_{j}}\|_{\varepsilon_{j}}\}$, \eqref{pallebis}, and Lemma \ref{lem:BP}, $f_j \in L^{q_0/(p-1)} (B(0,R))$ and
\begin{align*}
\left(\int_{B (0, R)} |\bar{\phi}_j \bar{v}_j^1|^{q_0/(p-1)} dz\right)^{(p-1)/q_0}
&\leq
C\left(\int_{B (0, R)} |\bar{\phi}_j \bar{v}_j^1|^2 dz\right)^{1/2}
\leq
C |\bar{\phi}_j |_{L^\infty(B (0, R))}\\
&\leq
C |\phi_{u_{\varepsilon_{j}}} (\exp_{P^1}(Q_{\varepsilon_{j}}^1+\varepsilon_{j}\cdot)) |_{L^\infty(B (-Q_{\varepsilon_{j}}^1/\varepsilon_{j}, r/(2\varepsilon_{j})))}\\
&\leq
C |\phi_{u_{\varepsilon_{j}}} (\exp_{P^1}(\cdot)) |_{L^\infty(B (0, r/2))}
\leq
C |\phi_{u_{\varepsilon_{j}}} |_{\infty}
\leq
C \|\phi_{u_{\varepsilon_{j}}} \|_{H^2}
\\
&
\leq
C |u_{\varepsilon_{j}} |_{2}^2
\leq
C \varepsilon_{j}^3.
\end{align*}
Hence, by \eqref{bootstrap},
\[
| \Delta \bar{v}_j^1 |_{L^{q_0/(p-1)} (B(0,R))} \leq C | f_j |_{L^{q_0/(p-1)} (B(0,R))} \leq C,
\]
and so, by a classical interpolation inequality,
$\bar{v}_j^1 \in W^{2, q_0/(p-1)} (B(0,R))$ and
\[
\|\bar{v}_j^1\|_{W^{2, 6/(p-1)} (B(0,R))} \leq C.
\]
If $q_0/(p-1)>3/2 $, namely if
\[
(p - 1) - 4 < 0,
\]
we get that $\bar{v}_j^1$ is continuous and, by the previous arguments, $\{\bar{v}_j^1\}$ is bounded in $C^{0,\alpha}(B(0,R))$ for some $\alpha\in (0,1)$.\\
If, instead, $q_0/(p-1) \leq 3/2$, namely if
\begin{equation*}
(p - 1) - 4 \geq 0,
\end{equation*}
or, equivalently, $ 5 \leq p < 6$, then $W^{2, q_0/(p-1)} (B(0,R))$ embeds in $L^{q_1} (B(0,R))$ with
\[ q_1 : = \frac{6}{(p - 1) - 4} . \]
Then we consider 
\[ \min \left\{ 2, \frac{q_1}{p - 1} \right\}  \]
and we iterate the procedure.
%
%
%
%
%
%
%
%
%
%
\\
So, at the $n$th step we take
\[ q_n : = \frac{6}{(p - 1)^n - 4 \sum_{k = 0}^{n - 1} (p - 1)^k} =
\frac{6}{(p - 1)^n - 4 \frac{(p - 1)^n - 1}{p - 2}} = \frac{6 (p - 2)}{(p -
	6) (p - 1)^n + 4}, \]
and we consider
\[ \min \left\{ 2, \frac{q_n}{p - 1} \right\} . \]
We can conclude if
\[ \min \left\{ 2, \frac{q_n}{p - 1} \right\}>\frac{3}{2} \]
which occurs in a finite number of steps since
\[ \frac{q_n}{p - 1} > \frac{3}{2},\]
namely for
\[ n > \frac{\log 4 - \log (6 - p)}{\log (p - 1)} - 1. \]
Observe that, at each step, whenever
\[ \min \left\{ 2, \frac{q_n}{p - 1} \right\}=\frac{q_n}{p - 1}>\frac{3}{2}, \]
arguing as before we get that $\{\bar{v}_j^1\}$ is bounded in $C^{0,\alpha}(B(0,R))$ for some $\alpha\in(0,1)$.\\
Now, let us write \eqref{bootstrap} as
\begin{equation}\label{bootstrap2}
	\begin{split}
		-g_{P^{1}}^{il}(Q_{\varepsilon_{j}}^{1}+\varepsilon_{j}z) \partial_{il} \bar{v}_j^1
		&- \frac{1}{|g_{P^1}(Q_{\varepsilon_{j}}^{1}+\varepsilon_{j}z)|^{1/2}}
		\partial_{l}\left(g_{P^{1}}^{il}(Q_{\varepsilon_{j}}^{1}+\varepsilon_{j}\cdot)|g_{P^{1}}(Q_{\varepsilon_{j}}^{1}+\varepsilon_{j}\cdot)|^{1/2}\right) (z)
		\partial_{i} \bar{v}_j^1\\
		&+\omega \bar{v}_j^1
		+ q^2 \bar{\phi}_j(z) \bar{v}_j^1
		= (\bar{v}_j^1)^{p-1}.
	\end{split}
\end{equation}
The continuity of $\bar{v}_j^1$ implies that the right hand side of \eqref{bootstrap2} is in $L^2(B(0,R))$. In addition, also $|\nabla (\bar{v}_j^1)^{p-1}| \in L^2(B(0,R))$ and so the right hand side of \eqref{bootstrap2} is in $H^1(B(0,R))$.
Thus, \cite[Theorem 8.10]{GT} implies that $\bar{v}_j^1\in W_{\rm loc}^{3,2}(B(0,2R/3))$ and so, by classical embeddings,
$\bar{v}_j^1\in C^{1,\alpha}(B(0,2R/3))$ for some $\alpha\in(0,1)$. Then, repeating the procedure we get that $\bar{v}_j^1\in C^{2,\alpha}(B(0,R/2))$ for some $\alpha\in(0,1)$ and, by Schauder estimate \cite[page 93]{GT},
\[
\|\bar{v}_{j}^{1}\|_{C^{2,\alpha}(B(0,R/4))}\leq C.
\]

\end{document}